\documentclass[11pt,a4paper]{amsart} 
\usepackage{amsmath,amsfonts,amssymb,amsthm}
\usepackage{graphicx}
\usepackage[hidelinks]{hyperref}
\usepackage{bbm}
\usepackage[foot]{amsaddr}

\calclayout

\newtheorem{assumption}{Assumption}
\newtheorem{defn}{Definition}
\newtheorem{theorem}{Theorem}

\newtheorem{lemma}{Lemma}

\usepackage{tikz}
\usetikzlibrary{decorations.pathreplacing,angles,quotes}

\usepackage{mathtools}

\DeclarePairedDelimiter\floor{\lfloor}{\rfloor}

\def\11{{\mathbbm 1}}
\def\S{{\mathcal S}}
\def\EE{{\mathbb E}}
\def\GGG{{\mathbf G}}
\def\BBB{{\mathbf B}}

\def\XXX{{\mathbf X}}
\def\ZZZ{{\mathbf Z}}

\def\AAA{{\mathbf A}}
\def\WWW{{\mathbf W}}
\def\PPP{{\mathbf P}}

\def\RR{{\mathbb R}}

\def\PP{\mathbb{P}}
\def\NN{\mathbb{N}}

\title{Convex function approximations for Markov decision processes}
\author{Jeremy Yee}
\email{jeremyyee@outlook.com.au}
\date{\today}

\begin{document}
\maketitle

\begin{abstract}
  This paper studies function approximation for finite horizon
  discrete time Markov decision processes under certain convexity
  assumptions. Uniform convergence of these approximations on compact
  sets is proved under several sampling schemes for the driving random
  variables. Under some conditions, these approximations form a
  monotone sequence of lower or upper bounding functions. Numerical
  experiments involving piecewise linear functions demonstrate that
  very tight bounding functions for the fair price of a Bermudan put
  option can be obtained with excellent speed (fractions of a cpu
  second). Results in this paper can be easily adapted to minimization
  problems involving concave Bellman functions.
\end{abstract}



\section{Introduction}
\label{sec_intro}

Sequential decision making under uncertainty can often be framed using
Markov decision processes (see \cite{hinderer,
  hernandezlerma_lasserreBook, puterman} and the references
within). However, due to the tedious nature of deriving analytical
solutions, some authors have suggested the use of approximate
solutions instead \cite{powell} and this view has been readily adopted
due to the advent of cheap and powerful computers. Typical numerical
methods either use a finite discretization of the state space
\cite{pages_etal2004, gray_neuhoff2006} or using a finite dimensional
approximation of the target functions
\cite{tsitsiklis_vanroy2001}. This paper will focus on the latter
approach for Markov decision processes containing only convex
functions and a finite number of actions. Convexity assumptions are
often used because it affords many theoretical benefits and the
literature is well developed \cite{rockafellar}. When there are a
finite number of actions, there are two main issues facing function
approximation methods in practice. The first involves estimating the
conditional expectation in the Bellman recursion. The second deals
with representing the reward functions and the expected value
functions using tractable objects. This paper approaches the
expectation operator using either an appropriate discretization of the
random variables or via Monte Carlo sampling.  The resulting Bellman
functions are then approximated using more tractable convex functions.
While many approaches have appeared in the literature \cite{birge1997,
  powell}, this paper differs from the usual in the following
manner. Typical approaches assume a countable state space or bounded
rewards/costs. However, many problems in practice do not satisfy this
and so this paper will not take this approach. This paper also
directly exploits convexity to extract desirable convergence
properties.  Given that decisions are often made at selected points in
time for many realistic problems, this paper assumes a discrete time
setting and this avoids the many technical details associated with
continuous time.

Regression based methods \cite{carriere1996,tsitsiklis_vanroy1999,
  longstaff_schwartz2001} have become a popular tool in the function
approximation approach. These methods represents the conditional
expected value functions as a linear combination of basis
functions. However, the choice of an appropriate regression basis is
often difficult and there are often unwanted oscillations in the
approximations.  This paper is related closest to the work done by
\cite{hernandezlerma_runggaldier1994} and \cite{hinz2014}. The authors
in \cite{hernandezlerma_runggaldier1994} considered monotone
increasing and decreasing bounding function approximations for
discrete time stochastic control problems using an appropriate
discretization of the random variable. Like this paper, they assumed
the functions in the Bellman recursion satisfy certain convexity
conditions. However, their functions are assumed to be bounded from
below unlike here. In their model, an action is chosen to minimise the
value as opposed to maximise in our setting. In addition, this paper
considers an extra layer of function approximation to represent the
resulting functions in the Bellman recursion. Nonetheless, this paper
adapts some of their brilliant insights.  In \cite{hinz2014}, the
author exploits convexity to approximate the value functions by using
convex piecewise linear functions formed using operations on the
tangents from the reward functions. The scheme in \cite{hinz2014}
results in uniform convergence on compact sets of the approximations
and has been successfully applied to many real world applications
\cite{hinz_yee_pomdp, hinz_yee_battery}. Unlike \cite{hinz2014}, this
paper will not impose global Lipschitz continuity on the functions in
the Bellman recursion and does not assume linear state dynamics. A
much more general class of function approximation is also studied
here. Moreover, this paper proves the same type of convergence under
random Monte Carlo sampling of the state disturbances as well as
deriving bounding functions. In this sense, this paper significantly
generalizes the remarkable work done by \cite{hinz2014}.

This paper is organized as follows. The next section introduces the
finite horzion discrete time Markov decision process setting and the
convexity assumptions used. The state is assumed to consist of a
discrete component and a continuous component.  This is a natural
assumption since it covers many practical applications. Convexity of
the target and approximating functions is assumed on the continuous
component of the state. A general convex function approximation scheme
is then presented in Section \ref{secApproximation}. In Section
\ref{secConvergence}, uniform convergence of these approximations to
the true value functions on compact sets is proved using
straightforward arguments. This convergence holds under various
sampling schemes for the random variables driving the state
evolution. Under conditions presented in Section \ref{secLower} and
Section \ref{secUpper}, these approximations form a non-decreasing
sequence of lower bounding or a non-increasing sequence of upper
bounding functions for the true value functions, respectively. This
approach is then demonstrated in Section \ref{secNumerical} using
piecewise linear function approximations for a Bermudan put
option. The numerical performance is impressive, both in terms of the
quality of the results and the computational times. Section
\ref{secConclusion} concludes this paper. Note that in this paper,
global Lipschitz continuity is referred to as simply Lipschitz
continuity for shorthand.


\section{Markov decision process}
\label{sec_problem}

Let $(\Omega, \mathcal{F},\PP)$ represent the probability space and
denote time by $t = 0, \dots, T$. The state is given by
$X_t := (P_t, Z_t)$ consisting of a discrete component $P_t$ taking
values in some finite set $\PPP$ and a continuous component $Z_t$
taking values in an open convex set $\ZZZ \subseteq \RR^d$. This paper
will refer to $\XXX := \PPP \times \ZZZ$ as the state space. Now at
each $t = 0, \dots, T - 1$, an action $a \in \AAA$ is chosen by the
agent from a finite set $\AAA$. Suppose the starting state is given by
$X_0 = x_0$ with probability one. Assume that $(P_{t})_{t=1}^T$
evolves as a controlled Markov chain with transition probabilities
$\alpha_{t+1}^{a, p,p'}$ for $a\in\AAA,p,p'\in\PPP,$ and
$t=0,\dots,T-1$. Here $\alpha_{t+1}^{a, p,p'}$ is the probability from
moving from $P_t=p$ to $P_{t+1}=p'$ after applying action $a$ at time
$t$. Assume the Markov process $(Z_{t})_{t=1}^T$ is governed by action
$a$ via
\begin{equation*}
Z_{t+1} = f_{t+1}(W^a_{t+1}, Z_t)
\end{equation*}
for some random variable
$W^a_{t+1}: \Omega \to \WWW \subseteq \RR^{d'}$ and measurable
transition function $f_{t+1}:\WWW \times \ZZZ \rightarrow \ZZZ$. At
time $t=0,\dots,T-1$, $W^a_{t+1}$ is the random disturbance driving
the state evolution after action $a$ is chosen. The random variables
$W^a_{t+1}$ and $f_{t+1}(W^a_{t+1},z)$ are assumed to be integrable
for $a\in\AAA$, $z\in\ZZZ$, and $t=0,\dots,T-1$.

The decision rule $\pi_{t}$ gives a mapping $\pi_{t}: \XXX \to \AAA$
which prescribes an action $\pi_{t}(x) \in \mathbf{A}$ for a given
state $x \in \XXX$. A sequence $\pi = (\pi_{t})_{t=0}^{T-1}$ of
decision rules is called a policy. For each starting state
$x_{0} \in \XXX$ and each policy $\pi$, there exists a probability
measure such that $\PP^{x_{0},\pi}(X_{0}=x_{0})=1$ and
\begin{equation*}
\PP^{x_{0}, \pi}(X_{t+1} \in \BBB \, | \, X_{0}, \dots, X_{t})=K^{\pi_{t}(X_{t})}_{t}(X_{t},\BBB )
\end{equation*}
for each measurable $\BBB \subset \XXX$ at $t=0,\dots, T-1$ where
$K^{a}_{t}$ denotes our Markov transition kernel after applying action
$a$ at time $t$. The reward at time $t=0,\dots,T-1$ is given by
$r_{t}: \XXX \times \AAA \to \RR$. A scrap value $r_{T}: \XXX \to \RR$
is collected at terminal time $t=T$. Given starting $x_{0}$, the controller's
goal is to maximize the expectation
\begin{equation*}
\label{policyvalue}
v_0^\pi(x_0) =\EE^{x_{0}, \pi}\left[\sum_{t=0}^{T-1} r_{t}(X_{t}, \pi_{t}(X_{t})) + r_{T}(X_{T})\right]
\end{equation*}
over all possible policies $\pi$. That is, to find an optimal policy
$\pi^{*}=(\pi^{*}_{t})_{t=0}^{T-1}$ satisfying
$v_0^{\pi^*}(x_0) \geq v_0^{\pi}(x_0)$ for any policy $\pi$. There are
only a finite number of possible policies in our setting.  Note that
this paper focuses only on Markov decision policies since history
dependent policies do not improve the above total expected rewards
\cite[Theorem 18.4]{hinderer}.

Let ${\mathcal K}^{a}_{t}$ represent the one-step transition operator
associated with the transition kernel. For each action $a \in \AAA$, the
operator ${\mathcal K}^{a}_{t}$ acts on functions
$v: \XXX \rightarrow \RR$ by
\begin{align}
  ({\mathcal K}^{a}_{t} v )(x) &=\int_{\XXX} v(x')K^{a}_{t}(x,dx')  \qquad \text{or} \nonumber \\
  ({\mathcal K}_{t}^{a}v)(p,z) &= \sum_{p'\in\PPP} \alpha_{t+1}^{a,p,p'}\EE[v(p',f_{t+1}(W^a_{t+1},z))].
  \label{transoperator}
\end{align}
If an optimal policy exists, it may be found using the following
dynamic programming principle. Introduce the Bellman operator
\begin{equation*}
{\mathcal T}_{t}v(p,z)=\max_{a \in \AAA} (r_{t}(p,z,a)+ {\mathcal K}^{a}_{t}v(p,z))
\label{1bel}
\end{equation*}
for $p\in\PPP$, $z\in\ZZZ$, and $t=0,\dots,T-1$. The resulting Bellman
recursion is given by
\begin{equation}
  v^{*}_{T}=r_{T}, \qquad v^{*}_{t}= {\mathcal T}_{t} v^{*}_{t+1} \label{original}
\end{equation}
for $t=T-1, \dots, 0$. If it exists, the solution
$(v^{*}_{t})_{t=0}^{T}$ gives the value functions and determines an
optimal policy $\pi^{*}$ by
\begin{equation}
  \pi^{*}_{t}(p,z)=\arg\max_{a \in\AAA}(r_{t}(p,z,a)+ {\mathcal
    K}_{t}^{a}v^{*}_{t+1}(p,z))
\label{optimal_policy}
\end{equation}
for $p \in \PPP$, $z\in\ZZZ$, and $t=0, \dots, T-1$.


\subsection{Convex value functions}

Since $\AAA$ is finite, the existence of the value functions is
guaranteed if the functions in the Bellman recursion are well
defined. This occurs, for example, when the reward, scrap, and
transition operator satisfies the following Lipschitz continuity. Let
$\| \cdot \|$ represent some norm below.

\begin{theorem}
  If $r_t(p,z,a)$, $r_T(p,z)$, and ${\mathcal K}_{t}^{a}\|z\|$ are
  Lipschitz continuous in $z$ for $a \in \AAA$, $p \in \PPP$, and
  $t = 0, \dots, T-1$, then the value functions exists.
\end{theorem}
\begin{proof} If the functions in the Bellman recursion have an upper
  bounding function \cite[Definition 2.4.1]{bauerle_rieder}, there
  will be no integrability issues \cite[Proposition
  2.4.2]{bauerle_rieder}. For $t =0, \dots, T-1$, $a \in \AAA$ and
  $p \in \PPP$, Lipschitz continuity yields
  \[ |r_t(p, z, a)| \leq |r_t(p, z', a)| + c\|z - z'\| \quad
    \text{and} \quad |r_T(p, z)| \leq |r_T(p, z')| + c\|z - z'\| \]
  for some constant $c$ and for $z,z'\in\ZZZ$. Therefore,
  $|r_t(p, z, a)| \leq c' (1 + \|z\|)$ and
  $|r_T(p, z)| \leq c' (1 + \|z\|)$ for some constant $c'$ for
  $p\in\PPP$ and $a\in\AAA$. By assumption on
  ${\mathcal K}_{t}^{a}\|z\|$, for $a \in \AAA$, $p\in\PPP$, and
  $t = 0, \dots, T-1$,
\[\bigg|\int_{\XXX} 1 + \| z'\| K^{a}_{t}((p,z),d(p',z'))\bigg| \leq c''(1 + \|z\|)\]
for some constant $c''$. Therefore, an upper bounding function for
our Markov decision process is given by $1 + \|z\|$ with constant
coefficient given by $\max\{c', c''\}$. \end{proof}

Note that on a finite dimensional vector space ($d<\infty$), all norms
are equivalent and so the choice of $\|\cdot\|$ is not particularly
important. The existence of the value functions and an optimal policy
is problem dependent and so the following simplifying assumption is
made to avoid any further technical diversions.

\begin{assumption} \label{assExist} Bellman functions $r_T(p,z)$,
  $r_t(p,z,a)$, and $\mathcal{K}^a_{t}v^*_{t+1}(p,z)$ are well defined
  and real-valued for all $p\in\PPP$, $z\in\ZZZ$, $a\in\AAA$, and
  $t=0,\dots,T-1$.
\end{assumption}

Under the above assumption, value functions $(v^*_t)_{t=0}^T$ and an
optimal policy $\pi^*$ exists. Further, they can be found via the
Bellman recursion presented earlier. Since convex functions plays a
central role in this paper and to avoid any potential misunderstanding
regarding this, the following definition is provided.

\begin{defn}
A function $h:\ZZZ\to\RR$ is convex in $z$ if
$$
h(\alpha z' + (1-\alpha)z'') \leq \alpha h(z') + (1-\alpha) h(z'') 
$$
for $z',z''\in\ZZZ$ and $0 \leq \alpha \leq 1$.
\end{defn}

Note that since $\ZZZ$ is open, convexity in $z$ automatically implies
continuity in $z$.  Now impose the following continuity and convexity
assumptions.

\begin{assumption} \label{assContinuity} Let $f_t(w,z)$ be continuous
  in $w$ for $z\in\ZZZ$ and $t=1,\dots,T$.
\end{assumption}

\begin{assumption} \label{assConvex} Let $r_t(p,z,a)$ and $r_T(p,z)$
  be convex in $z$ for $a \in \AAA$, $p \in \PPP$ and
  $t = 0, \dots, T-1$.  If $h(z)$ is convex in $z$, then
  $\mathcal{K}^a_th(z)$ is also convex in $z$ for $a\in\AAA$.
\end{assumption}

Assumption \ref{assContinuity} will be used to apply the continuous
mapping theorem in the proofs and Assumption \ref{assConvex}
guarantees that all the value functions $(v^{*}_{t}(p,z))_{t=0}^{T}$
are convex in $z$ for $t=0,\dots,T$ and $p\in\PPP$. Convexity can
be preserved by the transition operator (as stated in Assumption
\ref{assConvex}) due to the interaction between the value function and
the transition function. Let us explore just a few examples for $d=1$:
\begin{itemize}
\item If $v(p,z)$ is non-decreasing convex in $z$ and $f_{t+1}(w,z)$
  is convex in $z$, the composition $v(p, f_{t+1}(w,z))$ is convex in
  $z$.
\item Similarly, $v(p,f_{t+1}(w,z))$ is convex in $z$ if $v(p,z)$ is
  convex non-increasing in $z$ and $f_{t+1}(w,z)$ is concave in $z$.
\item If $f_{t+1}(w,z)$ is affine linear in $z$,
  ${\mathcal K}_{t}^{a}v(p, z)$ is also convex in $z$ if $v(p,z)$ is
  convex in $z$. Note this is a consequence of the first two cases
  since an affine linear function is both convex and concave.
\item If the space $\ZZZ$ can be partitioned so that in each component
  any of the above cases hold, the function composition
  $v(p,f_{t+1}(w,z))$ is also convex in $z$.
\end{itemize}


\section{Approximation} \label{secApproximation}

Denote $\mathbf{G}_{}^{(m)}\subset\ZZZ$ to be a $m$-point grid. This
paper adopts the convention that
$\mathbf{G}_{}^{(m)} \subset \mathbf{G}_{}^{(m+1)}$ and
$\cup_{m=1}^\infty \mathbf{G}_{}^{(m)}$ is dense in $\ZZZ$.  Suppose
$h:\ZZZ \rightarrow \RR$ is a continuous function and introduce some
function approximation operator ${\mathcal S}_{\mathbf{G}^{(m)}}$
dependent on the grid (this will be clarified shortly) which
approximates $h$ using another more tractable continuous function.
Some examples involving piecewise linear approximations are depicted
in Figure \ref{pic_schemes}.

\begin{figure}[h!]
  \centering
  \begin{tikzpicture}[yscale = 1.1, xscale=1.25]
    \draw [<->] (0,0) -- (6,0);	
    \draw [gray,dashed] (1,-.1) -- (1,4);
    \node [gray,below] at (1,0) {$g^{(i)}$};
    \draw [gray,dashed] (3,-.1) -- (3,4);
    \node [gray,below] at (3,0) {$g^{(i+1)}$};
    \draw [gray,dashed] (5,-.1) -- (5,4);
    \node [gray,below] at (5,0) {$g^{(i+2)}$};
    \draw (0,1) parabola (6,4);
    \draw [red] (0,0.9) -- (1.8,1.2);
    \draw [red] (1.8,1.2) -- (4,2.2);
    \draw [red] (4,2.2) -- (6,3.9);
    \draw [<->] (6.5,0) -- (12.5,0);
    \draw [gray,dashed] (7.5,-.1) -- (7.5,4);
    \node [gray,below] at (7.5,0) {$g^{(i)}$};
    \draw [gray,dashed] (9.5,-.1) -- (9.5,4);
    \node [gray,below] at (9.5,0) {$g^{(i+1)}$};
    \draw [gray,dashed] (11.5,-.1) -- (11.5,4);
    \node [gray,below] at (11.5,0) {$g^{(i+2)}$};
    \draw (6.5,3.5) .. controls (9,0) .. (12.5,3.5);
    \draw [red] (6.5,3.7) -- (7.5,2.15);
    \draw [red] (7.5,2.15) -- (9.5,0.95);
    \draw [red] (9.5,0.95) -- (11.5,2.5);
    \draw [red] (11.5,2.5) -- (12.5,3.7);
  \end{tikzpicture}
  \caption{Approximating the smooth target functions using piecewise
    linear functions.}
  \label{pic_schemes}
\end{figure}
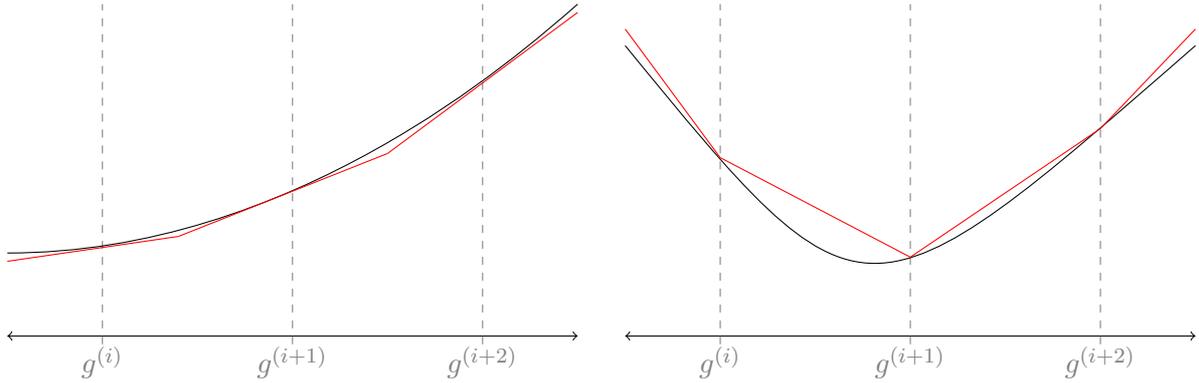

The first step in dealing with the Bellman recursion is to approximate
the transition operator \eqref{transoperator}. For each time
$t = 0, \dots, T-1$ and action $a \in \AAA$, choose a suitable
$n$-point disturbance sampling $(W^{a,(n)}_{t+1}(k))_{k=1}^{n}$ with
weights $(\rho^{a,(n)}_{t+1}(k))_{k=1}^{n}$.  Define the modified
transition operator by
\begin{equation}
{\mathcal K}^{a, (n)}_{t}v(p,z) = \sum_{p'\in\PPP}\alpha_{t+1}^{a,p,p'} \sum_{k=1}^{n}\rho^{a,(n)}_{t+1}(k) v(p', f_{t+1}(W^{a,(n)}_{t+1}(k),z))
\label{modKernel}
\end{equation}
and the modified Bellman operator by
\begin{equation} {\mathcal T}^{(m,n)}_{t}v(p, z) = \max_{a \in
    \AAA}\left(\mathcal{S}_{\mathbf{G}^{(m)}} {r_{t}(p, z, a)} +
    \mathcal{S}_{\mathbf{G}^{(m)}} {\mathcal K}^{a, (n)}_{t}v(p,
    z)\right)
\label{modBell}
\end{equation}
where $v(p,z)$ is a function continuous in $z$ for $p\in\PPP$. In the
above, the approximation scheme ${{\mathcal S}_{\mathbf{G}^{(m)}}}$ is
applied to the functions for each $p\in\PPP$ and $a\in\AAA$.  The
resulting backward induction
\begin{equation} 
  v^*_{T} = r_{T}, 
  \qquad v^{(m, n)}_{T-1} = {\mathcal T}^{(m, n)}_{T-1}{\mathcal S}_{\mathbf{G}^{(m)}} v^*_{T},
    \qquad v^{(m, n)}_{t} = {\mathcal T}^{(m, n)}_{t}v^{(m, n)}_{t+1}   \label{modBell1}
\end{equation}
for $t=T-2, \dots,0$ gives the modified value functions
$(v^{(m,n)}_{t})_{t=0}^{T-1}$. From Assumption \ref{assContinuity} and
the continuity of the modified value functions in $z$ for $p\in\PPP$,
it is clear that $v^{(m,n)}_t(p, f_t(w,z))$ is continuous in $w$ for
$p\in\PPP$, $z\in\ZZZ$, and $t=1,\dots,T-1$ since we have the
composition of continuous functions.

The central idea behind this paper is to approximate the original
problem in \eqref{original} with a more tractable problem given by
\eqref{modBell1}. Therefore, one would like the functions in modified
Bellman recursion to resemble their true counterparts. Since
Assumption \ref{assConvex} imposes convexity on the functions in the
original Bellman recursion and ${\mathcal S}_{\mathbf{G}^{(m)}}$ is used
to approximate these functions, the following assumption on convexity
and pointwise convergence on the dense grid is only natural.

\begin{assumption} \label{piecewiseS} For all convex functions
  $h:\ZZZ\to\RR$, suppose that ${\mathcal S}_{\mathbf{G}^{(m)}}h(z)$
  is convex in $z$ for $m\in\NN$ and that
  $\lim_{m\to\infty} {\mathcal S}_{\mathbf{G}^{(m)}}h(z) = h(z)$ for
  $z\in\cup_{m=1}^\infty \mathbf{G}_{}^{(m)}$.\end{assumption}

The following assumption is now made to hold throughout this paper to
guarantee the convexity of the modified value functions.

\begin{assumption} \label{assSampleConvex} Assume
  ${\mathcal K}^{a, (n)}_{T-1}{\mathcal S}_{\mathbf{G}^{(m)}} v^*_{T}(p, z)$ and
  ${\mathcal K}^{a, (n)}_{t}v_{t+1}^{(m,n)}(p, z)$ are convex in $z$ for
  $m,n\in\NN$, $a \in \AAA$, $p \in \PPP$ and $t=T-2, \dots,0$.
\end{assumption}

\begin{theorem} The functions in the modified Bellman recursion
  \eqref{modBell1} are convex in $z$ for $p\in\PPP$ under
  Assumptions \ref{assConvex}, \ref{piecewiseS} and
  \ref{assSampleConvex}.
\end{theorem}
\begin{proof} By assumption,
  $\mathcal{K}^{a,(n)}_{T-1}{\mathcal S}_{\mathbf{G}^{(m)}}v_T^*(p,z)$
  is convex in $z$ for $p\in\PPP$, $a\in\AAA$, and $m,n\in\NN$. The
  reward functions are convex in $z$ (by Assumption \ref{assConvex})
  for $p\in\PPP$ and $a\in\AAA$. Therefore,
  $\mathcal{S}_{\mathbf{G}^{(m)}} {r_{t}(p, z, a)} +
  \mathcal{S}_{\mathbf{G}^{(m)}} {\mathcal K}^{a, (n)}_{t}v(p, z)$ is
  convex in $z$ by Assumption \ref{piecewiseS}. The sum and pointwise
  maximum of convex functions is convex. Therefore,
  $v_{T-1}^{(m,n)}(p, z)$ is convex in $z$ for $p\in\PPP$ and
  $m,n\in\NN$ due to application of ${\mathcal
    S}_{\mathbf{G}^{(m)}}$. Proceeding inductively for
  $t = T-2, \dots, 0$ gives the desired result.
\end{proof}

Please note that the grid $\GGG^{(m)}$ can be easily made time
dependent without affecting the convergence results in the next
section. But for notational simplicity, this dependence is
omitted. Also note that the modified Bellman operator \eqref{modBell}
is not necessarily monotone since the operator
$\mathcal{S}_{\GGG^{(m}}$ is not necessarily monotone. It turns out
that the convergence results presented in the next section does not
require this property. However, to obtain lower and upper bounding
functions in Section \ref{secLower} and Section \ref{secUpper}, this
paper will impose Assumption \ref{assMonotone} in Section
\ref{secLower} to induce monotonicity in the modified Bellman
operator.


\section{Convergence} \label{secConvergence}

This section proves convergence of the modified value functions. There
are two natural choices for the disturbance sampling and weights:
\begin{itemize}
\item Use Monte Carlo to sample the disturbances randomly and the
  realizations are given equal weight or;
\item Partition $\WWW$ and use some derived value (e.g. the
  conditional averages) on each of the components for the
  sampling. The sampling weights are determined by the probability
  measure of each component.
\end{itemize}
While the first choice is easier to use and more practical in high
dimensional settings, the second selection confers many desirable
properties which will be examined later on.  First introduce the
following useful concepts which will be used extensively.

\begin{lemma}\label{lemmaRockafellar}
  Let $(h^{(n)})_{n\in\NN}$ be a sequence of real-valued convex
  functions on $\ZZZ$ i.e. $h^{(n)}:\ZZZ\to\RR$ for $n\in\NN$. If the
  sequence converges pointwise to $h$ on a dense subset of $\ZZZ$,
  then the sequence $(h^{(n)})_{n\in\NN}$ converges uniformly to $h$
  on all compact subsets of $\ZZZ$.
\end{lemma}
\begin{proof}
  See \cite[Theorem 10.8]{rockafellar}.
\end{proof}

\begin{defn}
  A sequence of convex real-valued functions $(h^{(n)})_{n\in\NN}$ on
  $\ZZZ$ is called a CCC (convex compactly converging) sequence in $z$
  if $(h^{(n)})_{n\in\NN}$ converges uniformly on all compact subsets of
  $\ZZZ$.
\end{defn}

In the following two subsections, let $p\in\PPP$, $a\in\AAA$, and
$t=0,\dots,T-1$ be arbitrary chosen. The sequence
$(v_{t+1}^{(n)}(p,z))_{n\in\NN}$ will be used to demonstrate the
behaviour of the modified value functions under the modified
transition operator. Assume $(v_{t+1}^{(n)}(p',z))_{n\in\NN}$ forms a
CCC sequence in $z$ converging to value functions $v^*_{t+1}(p',z)$
for all $p'\in\PPP$. Note that by Assumption \ref{assContinuity},
$v_{t+1}^{(n)}(p', f_{t+1}(w,z))$ is continuous in $w$ since we have a
composition of continuous functions.
 
\subsection{Monte Carlo sampling} \label{secMC}

The below establishes uniform convergence on compact sets under Monte
Carlo sampling.

\begin{theorem} \label{transConvergeMC} Let
  $(W^{a,(n)}_{t+1}(k))_{k=1}^{n}$ be idependently and identically
  distributed copies of $W^a_{t+1}$ and
  $\rho_{t+1}^{a,(n)}(k) = \frac{1}{n}$ for $k=1,\dots,n$. Assume
  these random variables reside on the same probability space as
  $W^a_{t+1}$. If $\WWW$ is compact, it holds that
  \begin{equation*} 
    \lim_{n\to\infty}
    \mathcal{K}_t^{a,(n)}v_{t+1}^{(n)}(p,z) =
    \mathcal{K}_t^{a}v_{t+1}^*(p,z), \quad z\in\ZZZ.
  \end{equation*}
  If $\mathcal{K}_t^{a,(n)}v_{t+1}^{(n)}(p,z)$ is also convex in $z$
  for $n\in\NN$, then
  $(\mathcal{K}_t^{a,(n)}v_{t+1}^{(n)}(p,z))_{n\in\NN}$ forms a CCC
  sequence in $z$.
\end{theorem}
\begin{proof} From the strong law of large numbers,
  $$
  \lim_{n\to\infty} \frac{1}{n}\sum_{k=1}^{n}
  v_{t+1}^*(p', f_{t+1}(W^{a,(n)}_{t+1}(k), z)) = \EE[v_{t+1}^*(p',
  f_{t+1}(W^{a}_{t+1}, z))]
  $$
  holds with probability one. The summands can be expressed
  as
  $$
  v_{t+1}^{(n)}(p', f_{t+1}(W^{a,(n)}_{t+1}(k), z)) +
  v_{t+1}^*(p', f_{t+1}(W^{a,(n)}_{t+1}(k), z)) -
  v_{t+1}^{(n)}(p', f_{t+1}(W^{a,(n)}_{t+1}(k), z)).
  $$
  Define
  $M_n:=\sup_{w\in\WWW} |v_{t+1}^*(p', f_{t+1}(w, z)) -
  v_{t+1}^{(n)}(p', f_{t+1}(w, z))|$. The continuity of
  $f_{t+1}(w, z)$ in $w$, the uniform convergence on compact sets of
  $v_{t+1}^{(n)}$ to $v_{t+1}^*$, and the compactness of $\WWW$ gives
  $ \lim_{n\to\infty} M_n = 0$. From Cesaro means,
  $$ 
  \lim_{n\to\infty}\frac{1}{n} \sum_{k=1}^n |v_{t+1}^*(p',
  f_{t+1}(W^{a,(n)}_{t+1}(k), z)) - v_{t+1}^{(n)}(p',
  f_{t+1}(W^{a,(n)}_{t+1}(k), z))| \leq \lim_{n\to\infty}\frac{1}{n}
  \sum_{k=1}^n M_k = 0
  $$
  with probability one and so
  $$
  \lim_{n\to\infty}
  \frac{1}{n}\sum_{k=1}^{n} v_{t+1}^{(n)}(p', f_{t+1}(W^{a,(n)}_{t+1}(k), z))
  = \EE[v_{t+1}^*(p', f_{t+1}(W^{a}_{t+1}, z))]
  $$
  with probability one. Therefore, 
  $$
  \lim_{n\to\infty} \sum_{p'\in\PPP}
  \alpha_{t+1}^{a,p,p'}\frac{1}{n}\sum_{k=1}^{n} v_{t+1}^{(n)}(p',
  f_{t+1}(W^{a,(n)}_{t+1}(k), z)) = \sum_{p'\in\PPP}
  \alpha_{t+1}^{a,p,p'}\EE[v_{t+1}^*(p', f_{t+1}(W^{a}_{t+1}, z))]
  $$
  almost surely and so the first part of the statement then follows.
  Now observe that the almost sure convergence in the first part of
  the statement holds for any choice of $z\in \ZZZ$. There are a
  countable number of $z\in \cup_{m\in\NN} \GGG^{(m)}$. A countable
  intersection of almost sure events is also almost sure. Therefore,
  $\mathcal{K}_t^{a,(n)}v_{t+1}^{(n)}(p,\cdot)$ converges to
  $\mathcal{K}_t^{a}v_{t+1}^*(p,\cdot)$ pointwise on a dense subset
  $\cup_{m\in\NN} \GGG^{(m)}$ of $\ZZZ$ with probability one. The
  second part of the statement then results from Lemma
  \ref{lemmaRockafellar}.
\end{proof}

The above assumes that $\WWW$ is compact.  While this compactness
assumption may seem problematic for unbounded $\WWW$ cases, one can
often find a compact subset $\overline\WWW\subset\WWW$ so obscenely
large that it contains the vast majority of the probability
mass. Therefore, from at least a numerical work perspective, the
drawback from this compactness is not that practically significant
especially considering computers typically have a limit on the size of
the numbers they can generate and because of machine epsilon. For
example, if $\WWW = \RR_+$, one can set $\overline \WWW$ where
$\max \overline\WWW$ is orders of magnitudes greater than this size
limit and $\min \overline \WWW$ is drastically smaller than the
machine epsilon. With this, one can then use Monte Carlo sampling in
practice as normal without restriction. The above convergence when
$\WWW$ is not compact will be addresssed in future research.

\subsection{Disturbance space partition}

This subsection proves the same convergence under partitioning of the
disturbance space $\WWW$. Introduce partition
$ \Pi^{(n)}=\{\Pi^{(n)}(k)\subset \WWW \, :\, k=1, \dots, n \}$ and 
define the diameter of the partition by
$$\delta^{(n)} := \max_{k=1,\dots,n} \sup\{\|w' - w''\| : w',w'' \in
\Pi^{(n)}(k)\}$$ if it exists. The case where $\WWW$ is compact is
considered first.

\begin{theorem} \label{transConvergeBounded} Suppose $\WWW$ is compact
  and let $\lim_{n\to\infty} \delta^{(n)} = 0$.  Choose sampling
  $(W^{a,(n)}_{t+1}(k))_{k=1}^{n}$ where
  $W^{a,(n)}_{t+1}(k) \in \Pi^{(n)} (k)$ and
  $\rho^{a,(n)}_{t+1}(k) = \PP(W^{a}_{t+1} \in \Pi^{(n)}(k))$ for
  $k=1,\dots,n$. It holds that
\begin{equation*}
  \lim_{n\to\infty} \mathcal{K}_t^{a,(n)}v_{t+1}^{(n)}(p,z) = \mathcal{K}_t^{a}v_{t+1}^*(p,z), \quad z\in\ZZZ. 
\end{equation*}
If $\mathcal{K}_t^{a,(n)}v_{t+1}^{(n)}(p,z)$ is also convex in $z$ for
$n\in\NN$, then $(\mathcal{K}_t^{a,(n)}v_{t+1}^{(n)}(p,z))_{n\in\NN}$
forms a CCC sequence in $z$.
\end{theorem}
\begin{proof} Denote $n$-point random variable
  $$
  W_{t+1}^{a,(n)} = \sum_{k=1}^{n} W_{t+1}^{a,(n)}(k) \11\left(W_{t+1}^{a}\in\Pi^{(n)}(k)\right)
  $$
  where $\11(\mathbf{B})$ denotes the indicator function of the set
  $\mathbf{B}$. Now
  $$
  \lim_{n\to\infty} \EE[ || W_{t+1}^{a,(n)} - W_{t+1}^{a} ||] \leq
  \lim_{n\to\infty} \delta^{(n)} = 0
  $$
  and so $W_{t+1}^{a,(n)}$ converges to $W_{t+1}^{a}$ in distribution
  as $n\to\infty$.  Using this convergence, the fact that $\WWW$ is
  compact, the fact that $v_{t+1}^{(n)}(p', f_{t+1}(w,z))$ and
  $v_{t+1}^{*}(p', f_{t+1}(w,z))$ are continuous in $w$, and the fact
  that $v_{t+1}^{(n)}$ converges to $v^*_{t+1}$ uniformly on compact
  sets, it can be seen that
  $$
  \lim_{n\to\infty}
  \EE[v_{t+1}^{(n)}(p', f_{t+1}(W^{a,(n)}_{t+1}, z))] =
  \EE[v^*_{t+1}(p', f_{t+1}(W^{a}_{t+1}, z))]
  $$
  for $p'\in\PPP$ and $z\in\ZZZ$. This first part of the statement
  then follows easily. The second part of the theorem follows from
  Lemma \ref{lemmaRockafellar}.
\end{proof}

The next theorem examines the case when $\WWW$ is not necessarily
compact. In addition, conditional averages are used for the
disturbance sampling and this is perhaps a more sensible choice given
that it minimizes the mean square error from the discretization of
$W^{a}_{t+1}$. In the following,
$\EE[W^{a}_{t+1} \mid W^{a}_{t+1} \in \Pi^{(n)}(k)]$ refers to the
expectation of $W^{a}_{t+1}$ conditioned on the event
$\{ W^{a}_{t+1} \in \Pi^{(n)}(k)\}$. This paper sometimes refers to
this as the local average.

\begin{theorem} \label{transConvergeLA} Suppose generated
  sigma-algebras
  $ \sigma^{(n)}_{a,{t+1}}=\sigma(\{W^a_{t+1} \in \Pi^{(n)}(k)\},
  \enspace k=1, \dots, n)$ satisfy
  $\sigma(W^a_{t+1}) = \sigma(\cup_{n\in\NN}\sigma^{(n)}_{a,{t+1}})$.
  Select sampling $(W^{a,(n)}_{t+1}(k))_{k=1}^{n}$ such that
  \[ W^{a,(n)}_{t+1}(k) = \EE[W^{a}_{t+1} \mid W^{a}_{t+1} \in
    \Pi^{(n)}(k)]\] with
  $\rho^{a,(n)}_{t+1}(k) = \PP(W^{a}_{t+1} \in \Pi^{(n)}(k))$ for
  $k=1,\dots,n$. If
  $(v_{t+1}^{(n)}(p', f_{t+1}(W^{a,(n)}_{t+1}, z)))_{n\in\NN}$ is
  uniformly integrable for $p'\in\PPP$ and $z\in\ZZZ$, then:
\begin{equation*}
  \lim_{n\to\infty} \mathcal{K}_t^{a,(n)}v_{t+1}^{(n)}(p,z) = \mathcal{K}_t^{a}v_{t+1}^*(p,z), \quad z\in\ZZZ.
\end{equation*}
If $\mathcal{K}_t^{a,(n)}v_{t+1}^{(n)}(p,z)$ is also convex in $z$ for
all $n\in\NN$, then
$(\mathcal{K}_t^{a,(n)}v_{t+1}^{(n)}(p,z))_{n\in\NN}$ forms a CCC
sequence in $z$.
\end{theorem}
\begin{proof} Denote random variable
  $ W_{t+1}^{a,(n)} = \EE[W_{t+1}^{a} \mid \sigma^{(n)}_{a,{t+1}}]$
  which takes values in the set of local averages
  $\left\{ \EE[W_{t+1}^{a} \mid W_{t+1}^{a} \in \Pi^{(n)}(k)] : k = 1,
    \dots, n \right\}$. On the set of paths where the almost sure
  convergence $W_{t+1}^{a,(n)} \to W_{t+1}^{a}$ holds by Levy's upward
  theorem \cite[Section 14.2]{williams}, the set
  $\{W_{t+1}^{a,(n)}: n\in\NN\}$ is bounded on each sample path since
  a convergent sequence is bounded and so
  \[\lim_{n\to\infty} v_{t+1}^{(n)}(p', f_{t+1}(W^{a,(n)}_{t+1}, z)) = v^*_{t+1}(p', f_{t+1}(W^a_{t+1}, z))\]
  with probability one since there is uniform convergence on bounded
  sets. Using the Vitali convergence theorem \cite[Section
  13.7]{williams},
  \[\lim_{n\to\infty} \EE[v_{t+1}^{(n)}(p',f_{t+1}(W^{a,(n)}_{t+1}, z))] = \EE[v^*_{t+1}(p',f_{t+1}(W^{a}_{t+1}, z))]\]
  for $p'\in\PPP$.  This proves the first part of the statement. The
  second part of the statement stems from Lemma
  \ref{lemmaRockafellar}.
\end{proof}

The above assumes that
$(v_{t+1}^{(n)}(p', f_{t+1}(W^{a,(n)}_{t+1}, z)))_{n\in\NN}$ is
uniformly integrable. This is satisfied when $\WWW$ is compact. For
$\WWW$ not compact, Theorem \ref{lipschitzUI2} may be useful.

\begin{lemma} \label{lemmaUIcond} Let random variable $Y$ be
  integrable on $(\Omega, \mathcal{F},\PP)$. The class of random
  variables
  $ (\EE[Y \mid \mathcal{G}] : \mathcal{G} \text{ is a
    sub-sigma-algebra of } \mathcal{F}) $ is uniformaly integrable.
\end{lemma}
\begin{proof}
  See \cite[Section 13.4]{williams}.
\end{proof}

In the following, the function $f_{t+1}(w,z)$ is said to be convex in
$w$ component-wise if each component of $f_{t+1}(w,z)$ is convex in
$w$.

\begin{theorem} \label{lipschitzUI2} Let
  $W_{t+1}^{a,(n)} = \EE[W_{t+1}^{a} \mid \sigma^{(n)}_{a,{t+1}}]$ and
  $v_{t+1}^{(n)}(p',z)$ be Lipschitz continuous in $z$ with Lipschitz
  constant $c_n$. If $\sup_{n\in\NN} c_n < \infty$ and for $z\in\ZZZ$
  either:
  \begin{itemize}
  \item $f_{t+1}(w,z)$ is Lipschitz continuous in $w$, or
  \item $\|f_{t+1}(w,z)\|$ is convex in $w$, or
  \item $f_{t+1}(w,z)$ is convex in $w$ component-wise,
  \end{itemize}
  holds, then
  $(v_{t+1}^{(n)}(p', f_{t+1}(W^{a,(n)}_{t+1}, z)))_{n\in\NN}$ is
  uniformly integrable for $z\in\ZZZ$.
\end{theorem}
\begin{proof}
  From Lipschitz continuity,
  $$
  |v_{t+1}^{(n)}(p', f_{t+1}(W^{a,(n)}_{t+1}, z))| \leq
  |v^{(n)}_{t+1}(p', z')| + c_n\|f_{t+1}(W^{a,(n)}_{t+1},z) - z'\|
  $$
  holds for $z'\in\ZZZ$. Since $v_{t+1}^{(n)}$ converges to
  $v^*_{t+1}$, it is enough to verify that
  $(c_n \| f_{t+1}(W^{a,(n)}_{t+1},z)-z'\|)_{n\in\NN}$ is uniformly
  integrable to prove the above statement. Now if $f_{t+1}(w,z)$ is
  Lipschitz continuous in $w$,
  \begin{equation} \label{aa}
  \|f_{t+1}(W^{a,(n)}_{t+1},z)\| \leq \|f_{t+1}(w', z)\| + c\| \EE[W^{a}_{t+1}\mid \sigma^{(n)}_{a,t+1}] - w'\|.
  \end{equation}
  for some constant $c$ and $w'\in\WWW$. Now suppose instead that
  $||f_{t+1}(w,z)||$ is convex in $w$. We know from Jensen's
  inequality that
  \begin{equation} \label{bb}
  \|f_{t+1}(W^{a,(n)}_{t+1},z)\| \leq \EE[ \|f_{t+1}(W^a_{t+1},z)\| \mid
  \sigma^{(n)}_{a,t+1}].
  \end{equation}
  Finally, if $f_{t+1}(w,z)$ is convex in $w$ component-wise, Jensen's
  gives
  $$
  f_{t+1}(W^{a,(n)}_{t+1},z) \leq \EE[ f_{t+1}(W^a_{t+1},z) \mid
  \sigma^{(n)}_{a,t+1}]
  $$
  holding component-wise. From the above inequality and the fact that
  convex functions are bounded below by an affine linear function
  (e.g. tangents), the following holds component-wise for some
  constants $b$ and $c'$:
  \begin{equation} \label{cc} |f_{t+1}(W^{a,(n)}_{t+1},z)| \leq |\EE[
    f_{t+1}(W^a_{t+1},z) \mid \sigma^{(n)}_{a,t+1}]| + |b + c'
    \EE[W^a_{t+1} \mid \sigma^{(n)}_{a,t+1}]|.
  \end{equation}
  Using Lemma \ref{lemmaUIcond} and $\sup_{n\in\NN} c_n < \infty$,
  Equations \eqref{aa}, \eqref{bb} or \eqref{cc} reveals that
  $(c_n \|f_{t+1}(W^{a,(n)}_{t+1},z) - z'\|)_{n\in\NN}$ is uniformly
  integrable for $z'\in\ZZZ$ because it is dominated by a family of
  uniformly integrable random variables.
\end{proof}

The above generalizes the condition used by \cite{hinz2014} to ensure
uniform integrability his approximation scheme. Before proceeding,
note that $\Pi^{(n)}$ can be made time dependent without affecting the
convergence above.

\subsection{Modified value functions}

The following establishes the uniform convergence on compact sets of
the resulting modified value functions under each of the disturbance
sampling methods. Let $(m_n)_{n\in\NN}$ and $(n_m)_{m\in\NN}$ be
sequences of natural numbers increasing in $n$ and $m$, respectively.

\begin{lemma} \label{lemmaCCC} Suppose $(h_1^{(n)})_{n\in\NN}$ and
  $(h_2^{(n)})_{n\in\NN}$ are CCC sequences on $\ZZZ$ converging to
  $h_1$ and $h_2$, respectively. Define
  $h^{(n)}_3(z) := \max (h^{(n)}_1(z), h^{(n)}_2(z))$ and
  $h_3(z) := \max (h_1(z), h_2(z))$. Then:
  \begin{itemize}
  \item $(h_1^{(n)} + h_2^{(n)})_{n\in\NN}$ is a CCC sequences on
    $\ZZZ$ converging to $h_1 + h_2$,
  \item $(\mathcal{S}_{\GGG^{m_n}}h_1^{(n)})_{n\in\NN}$ is a CCC
    sequences on $\ZZZ$ converging to $h_1$, and
  \item $(h_3^{(n)})_{n\in\NN}$ is a CCC sequences on $\ZZZ$ converging
    to $h_3$.
  \end{itemize}
\end{lemma}
\begin{proof}
  They can be proved easily using the definition of uniform
  convergence on compact sets, Assumption \ref{piecewiseS}, and Lemma
  \ref{lemmaRockafellar}.
\end{proof}

\begin{theorem} \label{valueConvergePW} The sampling in Theorem
  \ref{transConvergeMC}, Theorem \ref{transConvergeBounded} or Theorem
  \ref{transConvergeLA} gives
  \[\lim_{n\to\infty} v_t^{(m_{n},n)}(p,z) = \lim_{m\to\infty} v_t^{(m,n_m)}(p,z) = v^*_t(p,z)\]
  for $p\in\PPP$, $z\in\ZZZ$, and $t=T-1,\dots,0$. Also,
  $(v_t^{(m_{n},n)}(p,z))_{n\in\NN}$ and
  $(v_t^{(m,n_m)}(p,z))_{m\in\NN}$ both form CCC sequences in $z$ for
  all $p\in\PPP$ and $t=T-1,\dots,0$.
\end{theorem}
\begin{proof} Let us consider the limit as ${n\to\infty}$ first and
  prove this via backward induction. At $t = T-1$, Lemma
  \ref{lemmaCCC} reveals that
  $(\mathcal{S}_{\GGG^{(m_n)}} v^*_T(p,z))_{n\in\NN}$ forms a CCC
  sequence in $z$ for $p\in\PPP$ and converges to $v^*_T(p,z)$. Now
  from Assumption \ref{assSampleConvex},
  $\mathcal{K}_{T-1}^{a,(n)} \mathcal{S}_{\GGG^{(m_n)}} v^*_T(p,z)$ is
  convex in $z$. Therefore, using Assumption \ref{assSampleConvex} and
  either Theorem \ref{transConvergeMC}, Theorem
  \ref{transConvergeBounded} or Theorem \ref{transConvergeLA} reveals
  that
  $(\mathcal{K}_{T-1}^{a,(n)}
  \mathcal{S}_{\GGG^{(m_n)}}v^*_T(p,z))_{n\in\NN}$ forms a CCC
  sequence in $z$ converging to $\mathcal{K}_{T-1}^{a}
  v^*_T(p,z)$. From the above and Lemma \ref{lemmaCCC}, we know that
  $(\mathcal{S}_{\GGG^{(m_n)}}r_{T-1}(p,z,a) +
  \mathcal{S}_{\GGG^{(m_n)}} \mathcal{K}_{T-1}^{a,(n)}
  \mathcal{S}_{\GGG^{(m_n)}} v^*_T(p,z))_{n\in\NN}$ forms a CCC
  sequence in $z$ and converges to
  $r_{T-1}(p,z,a) + \mathcal{K}_{T-1}^{a} v^*_T(p,z)$. Since $\AAA$ is
  finite, Lemma \ref{lemmaCCC} implies that
  $(v_{T-1}^{(m_n,n)}(p,z))_{n\in\NN}$ forms a CCC sequence in $z$ and
  converges to $v_{T-1}^*(p,z)$ for $p\in\PPP$.  At $t=T-2$, it can be
  shown using the same logic above for $p\in\PPP$ and $a\in\AAA$ that
  $(\mathcal{K}_{T-2}^{a,(n)}v^{(m_n,n)}_{T-1}(p,z))_{n\in\NN}$ forms
  a CCC sequence in $z$ and converges to
  $\mathcal{K}_{T-2}^{a} v^*_{T-1}(p,z)$. Following the same lines of
  argument above eventually leads to
  $(v_{T-2}^{(m_n,n)}(p,z))_{n\in\NN}$ forming a CCC sequence in $z$
  and that it converges to $v^*_{T-2}(p,z)$ for $p\in\PPP$.
  Proceeding inductively for $t = T-3, \dots, 0$ gives the desired
  result.  The proof for the ${m\to\infty}$ case follows the same
  lines as above. \end{proof}


\section{Lower bounds} \label{secLower}

Observe that the convergence results presented so far does require the
modified Bellman operator \eqref{modBell} to be a monotone
operator. However, this is needed to obtain lower and upper bounding
functions.

\begin{assumption} \label{assMonotone} For all convex functions
  $h',h'':\ZZZ \rightarrow \RR$ where $h'(z) \leq h''(z)$ for
  $z\in\ZZZ$, assume
  ${\mathcal S}_{\GGG^{(m)}}h'(z) \leq {\mathcal
    S}_{\GGG^{(m)}}h''(z)$ for $z\in\ZZZ$ and $m\in\NN$.
\end{assumption}

The modified Bellman operator \eqref{modBell} is now monotone i.e. for
$m,n\in\NN$, $p\in\PPP$, $z\in\ZZZ$, and $t=0,\dots,T-1$, we have
$\mathcal{T}_t^{(m,n)}v'(p,z) \leq \mathcal{T}_t^{(m,n)}v''(p,z)$ if
$v'(p',z) \leq v''(p',z)$ for $p'\in\PPP$.  This stems from the
monotonicity of \eqref{modKernel} and Assumption \ref{assMonotone}.
Under the following conditions, the modified value functions
constructed using the disturbance sampling in Theorem
\ref{transConvergeLA} leads to a non-decreasing sequence of lower
bounding functions. Partition $\Pi^{(n+1)}$ is said to refine
$\Pi^{(n)}$ if each component in $\Pi^{(n+1)}$ is a subset of a
component in $\Pi^{(n)}$.

\begin{lemma} \label{lemmaK} Let $t=0,\dots,T-1$ and
  $v(p,f_{t+1}(w,z))$ be convex in $w$ for all $p\in\PPP$. If
  $\Pi^{(n+1)}$ refines $\Pi^{(n)}$, then Theorem
  \ref{transConvergeLA} gives
  $ \mathcal{K}_{t}^{a,(n)}v(p,z) \leq
  \mathcal{K}_{t}^{a,(n+1)}v(p,z).  $
\end{lemma}
\begin{proof}
  Without loss of generality, assume the last two components of
  $\Pi^{(n+1)}$ are both subsets of the last component of
  $\Pi^{(n)}$. With this,
  \begin{multline*}
    \mathcal{K}^{a, (n)}_{t} v(p,z) - \sum_{p'\in\PPP}
    \alpha_{t+1}^{a,p,p'} \rho^{a,(n)}_{t+1}(n)
    v(p',f_{t+1}(W^{a,(n)}_{t+1}(n),z)) \\ = \mathcal{K}^{a,
      (n+1)}_{t} v(p,z) - \sum_{p'\in\PPP} \alpha_{t+1}^{a,p,p'}
    \sum_{k=n}^{n+1} \rho^{a,(n+1)}_{t+1}(k)
    v(p',f_{t+1}(W^{a,(n+1)}_{t+1}(k),z)).
  \end{multline*}
  Since $v(p',f_{t+1}(w,z))$ is convex in $w$, it holds that
  \begin{multline*}
    v(p',f_{t+1}(W^{a,(n)}_{t+1}(n),z))  \leq
    \frac{\rho^{a,(n+1)}_{t+1}(n)}{ \rho^{a,(n)}_{t+1}(n)}
    v(p',f_{t+1}(W^{a,(n+1)}_{t+1}(n),z)) \\ +
    \frac{\rho^{a,(n+1)}_{t+1}(n+1)}{
      \rho^{a,(n)}_{t+1}(n)}v(p',f_{t+1}(W^{a,(n+1)}_{t+1}(n+1),z))
  \end{multline*}
  because
  $$
  W^{a,(n)}_{t+1}(n) = \frac{\rho^{a,(n+1)}_{t+1}(n)}{ \rho^{a,(n)}_{t+1}(n)}
  W^{a,(n+1)}_{t+1}(n) + \frac{\rho^{a,(n+1)}_{t+1}(n+1)}{ \rho^{a,(n)}_{t+1}(n)}
  W^{a,(n+1)}_{t+1}(n+1)
  $$
  and
  $\rho^{a,(n+1)}_{t+1}(n) + \rho^{a,(n+1)}_{t+1}(n+1) = \rho^{a,(n)}_{t+1}(n)$.
  Therefore, for all $p'\in\PPP$
  $$
  \rho^{a,(n)}_{t+1}(n) v(p',f_{t+1}(W^{a,(n)}_{t+1}(n),z)) \leq \sum_{k=n}^{n+1}
  \rho^{a,(n+1)}_{t+1}(k) v(p',f_{t+1}(W^{a,(n+1)}_{t+1}(k),z))
  $$
  and so
  $\mathcal{K}^{a, (n)}_{t} v(p,z) \leq \mathcal{K}^{a, (n+1)}_{t}
  v(p,z)$ as claimed. \qed
\end{proof}

\begin{theorem}\label{lowerBound} 
  Using Theorem \ref{transConvergeLA} gives for $p\in\PPP$,
  $z\in\ZZZ$, $m,n\in\NN$, and $t=0,\dots,T-1$:
  \begin{itemize}
  \item $v_t^{(m,n)}(p,z) \leq v_t^*(p,z)$ if
    $v^*_{t'}(p', f_{t'}(w, z'))$ is convex in $w$ and if
    ${\mathcal S}_{\GGG^{(m')}}h(z') \leq h(z')$ for $p'\in\PPP$,
    $z'\in\ZZZ$, $t'=1,\dots,T$, and all convex functions $h$.
  \item $v_t^{(m,n)}(p,z) \leq v_t^{(m,n+1)}(p,z)$ when
    $\Pi^{(n'+1)}$ refines $\Pi^{(n')}$ and if
    $\mathcal{S}_{\GGG^{(m')}}v^*_{T}(p', f_{T}(w, z'))$ and
    $v^{(m',n')}_{t'}(p', f_{t'}(w, z'))$ are convex in $w$ for
    $m',n'\in\NN$, $p'\in\PPP$, $z'\in\ZZZ$, and $t'=1,\dots,T-1$.
  \item $v_t^{(m,n)}(p,z) \leq v_t^{(m+1,n)}(p,z)$ if
    ${\mathcal S}_{\GGG^{(m')}}h(z') \leq {\mathcal
      S}_{\GGG^{(m'+1)}}h(z')$ for $z'\in\ZZZ$, $m'\in\NN$, and all
    convex functions $h$.
  \end{itemize}
\end{theorem}
\begin{proof} The three inequalities are proven separately using
  backward induction.

  1) Recall that
  $W_{t+1}^{a,(n)} = \EE[W_{t+1}^{a} \mid \sigma^{(n)}_{a,{t+1}}]$ for
  $a\in\AAA$ and $t=0,\dots,T-1$.. From the tower property, Jensen's
  inequality and the monotonicity of \eqref{modKernel}:
  $$
    \EE[v_T^*(p', f_{T}(W_{T}^{a},z'))]
    \geq \EE[v_T^*(p', f_{T}(W_{T}^{a,(n)},z'))] 
    \geq \sum_{k=1}^n \rho_T^{a,(n)}(k) \mathcal{S}_{\GGG^{(m)}} v^*_T(p', f_T(W^{a,(n)}_T(k), z'))
  $$
  for all $p'\in\PPP$ and $z'\in\ZZZ$. Therefore,
  $\mathcal{K}^{a,(n)}_{T-1}\mathcal{S}_{\GGG^{(m)}} v^*_T(p,z) \leq
  \mathcal{K}^{a}_{T-1}v^*_T(p,z)$ which in turn implies that
  $v_{T-1}^{(m,n)}(p,z) \leq v_{T-1}^*(p,z)$ since
  ${\mathcal S}_{\GGG^{(m)}}r_{T-1}(p,z,a) + {\mathcal
    S}_{\GGG^{(m)}}\mathcal{K}^{a,(n)}_{T-1}\mathcal{S}_{\GGG^{(m)}}
  v^*_T(p,z) \leq r_{T-1}(p,z,a)+ \mathcal{K}^{a}_{T-1}v^*_T(p,z)$.
  Proceeding inductively for $t=T-2,\dots,0$ using a similar argument
  as above gives the first inequality in the statement.

  2) Lemma \ref{lemmaK} gives
  $\mathcal{K}^{a, (n)}_{T-1}\mathcal{S}_{\GGG^{(m)}} v^{*}_T(p,z)
  \leq \mathcal{K}^{a, (n+1)}_{T-1}\mathcal{S}_{\GGG^{(m)}}
  v^{*}_T(p,z)$ which implies
  $v^{(m,n)}_{T-1}(x) \leq v^{(m,n+1)}_{T-1}(p,z)$. Using the
  monotonicity of \eqref{modKernel} and a similar argument as above,
  one can show
  $$
  \mathcal{K}^{a, (n)}_{T-2}v^{(m,n)}_{T-1}(p,z) \leq \mathcal{K}^{a,
    (n)}_{T-2}v^{(m,n+1)}_{T-1}(p,z) \leq \mathcal{K}^{a,
    (n+1)}_{T-2}v^{(m,n+1)}_{T-1}(p,z)
  $$ 
  $\implies v^{(m,n)}_{T-2}(p,z) \leq v^{(m,n+1)}_{T-2}(p,z)$. Proceeding
  inductively for $t=T-3,\dots,0$ proves the desired result.

  3) This can be easily proved by backward induction using the
  monotonicty of \eqref{modKernel} and by the fact that
  ${\mathcal S}_{\GGG^{(m)}}r_t(p,z,a) \leq {\mathcal
    S}_{\GGG^{(m+1)}}r_t(p,z,a)$ and
  ${\mathcal S}_{\GGG^{(m)}}r_T(p,z) \leq {\mathcal
    S}_{\GGG^{(m+1)}}r_T(p,z)$ for $p\in\PPP$, $z\in\ZZZ$, $a\in\AAA$,
  $m\in\NN$, and $t=0,\dots,T-1$.
\end{proof}

It turns out that if the modified value functions are bounded above by
the true value functions such as in the first case point of Theorem
\ref{lowerBound}, the uniform integrability assumption in Theorem
\ref{transConvergeLA} holds automatically. This is proved in the next
subsection.

\subsection{Uniform integrability
  condition} \label{secUIConvexity} 

Theorem \ref{theoremUI} below differs from Theorem \ref{lipschitzUI2}
in that Lipschitz continuity in $z$ is not assumed for the
approximating functions.  Instead, the value functions are assumed to
bound the approximating functions from above.  In the following, the
sequence of functions $(v_{t+1}^{(n)}(p,z))_{n\in\NN}$ from the
previous section is reused.  Recall that this sequence converges
uniformly to $v^*_{t+1}$ on compact sets.

\begin{lemma} \label{lemmaSL} Suppose that
  $v_{t+1}^{(n)}(p,z) \leq v_{t+1}^{*}(p,z)$ for all $n\in\NN$,
  $p\in\PPP$, and $z\in\ZZZ$. For fixed $p\in\PPP$ and all
  $z'\in\ZZZ$, there exists constants $c_{n,z'}\in\RR^d$ such that
  $$
  v_{t+1}^{*}(p, z) \geq v_{t+1}^{(n)}(p, z) \geq v_{t+1}^{(n)}(p,z')
  + c_{n,z'}^T (z' - z)
  $$
  for $z\in\ZZZ$ and $n\in\NN$. It also holds that
  $\sup_{n\in\NN} \| c_{n,z'}\| < \infty$.
\end{lemma}
\begin{proof}
  The first part follows from the definition of a tangent for a convex
  function. Now note that if $\sup_{n\in\NN} \| c_{n,z'}\|$ is not
  bounded for $z'\in\ZZZ$, then there exists $\widehat n$ and
  $\widehat z$ where
  $$
  v_{t+1}^{*}(p, \widehat z) \leq v_{t+1}^{(\widehat n)}(p,z') + c_{\widehat n, \widehat z}^T (z' - \widehat z)
  $$
  since $v^{(n)}_{t+1}$ converges to $v^*_{t+1}$. This yields a
  contradiction.
\end{proof}

\begin{theorem} \label{theoremUI} Suppose
  $v_{t+1}^{(n)}(p,z) \leq v_{t+1}^{*}(p,z)$ for $n\in\NN$,
  $p\in\PPP$, and $z\in\ZZZ$. Assume $v^{*}_{t+1}(p, f_{t+1}(w, z))$
  is convex in $w$ and let
  $W_{t+1}^{a,(n)} = \EE[W_{t+1}^a \mid \sigma^{(n)}_{a,{t+1}}]$.  If
  for $z\in\ZZZ$ either:
  \begin{itemize}
  \item $f_{t+1}(w,z)$ is Lipschitz continuous in $w$,
  \item $\|f_{t+1}(w,z)\|$ is convex in $w$, or
  \item $f_{t+1}(w,z)$ is convex in $w$ component-wise,
  \end{itemize}
  holds, then
  $(v_{t+1}^{(n)}(p,f_{t+1}(W_{t+1}^{a,(n)}, z)))_{n\in\NN}$ is
  uniformly intergrable.
\end{theorem}
\begin{proof}
  Jensen's inequality and $v_{t+1}^{(n)} \leq v_{t+1}^*$ gives
  $$
  v_{t+1}^{(n)}(p, f_{t+1}(W^{a,(n)}_{t+1}, z)) \leq v_{t+1}^{*}(p,
  f_{t+1}(W^{a,(n)}_{t+1}, z)) \leq \EE [v_{t+1}^{*}(p,f_{t+1}
  (W^{a}_{t+1}, z)) \mid \sigma^{(n)}_{a,{t+1}}].
  $$
  Now from Lemma \ref{lemmaSL}, there exists constants $c_{n,z'}\in\RR^d$
  such that for all $n\in\NN$:
  $$
  v_{t+1}^{(n)}(p, f_{t+1}(W^{a,(n)}_{t+1}, z)) \geq v_{t+1}^{(n)}(p,z') + c_n^T
  (z' - f_{t+1}(W^{a,(n)}_{t+1},z))
  $$
  for some $z'\in\ZZZ$ with probability one. Using the above
  inequalities: 
  \begin{multline*}
    |v_{t+1}^{(n)}(p, f_{t+1}(W^{a,(n)}_{t+1}, z))| \leq |\EE
    [v_{t+1}^{*}(p, f_{t+1} (W^{a}_{t+1}, z)) \mid
    \sigma^{(n)}_{a,{t+1}}]| \\ + |v_{t+1}^{(n)}(p,z') + c_{n,z'}^T (z' -
    f_{t+1}(W^{a,(n)}_{t+1},z)) |
  \end{multline*}
  almost surely.  The first term on the right forms a uniformly
  integrable family of random variables from Lemma
  \ref{lemmaUIcond}. Since $v_{t+1}^{(n)}$ converges to $v_{t+1}^*$,
  $\sup_{n\in\NN} v_{t+1}^{(n)}(p,z') < \infty$. From Lemma
  \ref{lemmaSL}, $\sup_{n\in\NN} \|c_{n,z'}\| < \infty$.  Now
  $(f_{t+1}(W^{a,(n)}_{t+1},z))_{n\in\NN}$ can be shown to be
  uniformly integrable using a similar argument as in the proof of
  Theorem \ref{lipschitzUI2}. Therefore,
  $(v_{t+1}^{(n)}(p, f_{t+1}(W_{t+1}^{a,(n)}, z)))_{n\in\NN}$ is
  dominated by a family of uniformly integrable random variables and
  so is also uniformly integrable.
\end{proof}


\section{Upper bounds} \label{secUpper}

For the case where $\WWW$ is compact, a sequence of upper bounding
function approximations can also be constructed using the following
setting from \cite{hernandezlerma_runggaldier1994}. It is well known
that a convex function continuous on a compact convex set attains its
maximum at an extreme point of that set. The following performs a type
of averaging of these extreme points to obtain an upper bound using
Jensen's inequality (see \cite[Section 5]{birge_wets1986}). Let us
assume that the closures of each of the components in partition
$\Pi^{(n)}$ are convex and contain a finite number of extreme
points. Denote the set of extreme points of the closure
$\widetilde\Pi^{(n)}(k)$ of each component $\Pi^{(n)}(k)$ by
$$
\mathcal{E}(\widetilde\Pi^{(n)}(k)) = \{e^{(n)}_{k,i} : i = 1, \dots, L_k^{(n)}\}
$$
where $L_k^{(n)}$ is the number of extreme points in
$\widetilde\Pi^{(n)}(k)$. Suppose there exists weighting functions
$q^{a,(n)}_{t+1,k,i}:\WWW\to[0,1]$ satisfying
\begin{equation*}
  \sum_{i=1}^{L^{(n)}_k}q^{a,(n)}_{t+1,k,i}(w) = 1 
  \quad \text{and} \quad
  \sum_{i=1}^{L^{(n)}_k}q^{a,(n)}_{t+1,k,i}(w) e^{(n)}_{k,i} = w 
\end{equation*}
for $w\in\widetilde\Pi^{(n)}(k)$ and $k=1,\dots,n$. Suppose
$\rho^{a,(n)}_{t+1}(k) = \PP(W^{a}_{t+1} \in \Pi^{(n)}(k)) > 0$ for
$k=1,\dots,n$ and define random variables
$\overline{W}_{t+1,k}^{a,(n)}$ satisfying
$$
\PP\left(\overline{W}_{t+1,k}^{a,(n)}=e^{(n)}_{k,i}\right) =
\frac{\bar q^{a,(n)}_{t+1, k,i}}{\rho^{a,(n)}_{t+1}(k)} \quad
\text{where} \quad \bar q^{a,(n)}_{t+1, k,i} =
\int_{\widetilde\Pi^{(n)}(k)} q^{a,(n)}_{t+1,k,i}(w)
\mu^{a}_{t+1}(\mathrm{d}w)
$$
and $\mu^{a}_{t+1}(\mathbf{B}) = \PP(W^{a}_{t+1} \in
\mathbf{B})$. To grasp the intuition for the upper bound, note that if
$g:\WWW \to \RR$ is a convex and continuous function, then
$$
\EE[g(W^a_{t+1})\11(W^a_{t+1} \in \Pi^{(n)}(k))] \leq
\sum_{i=1}^{L_k^{(n)}} \bar q^{a,(n)}_{t+1, k,i} g(e^{(n)}_{k,i})
$$
for $k=1,\dots,n$ (see \cite[Corollary
7.2]{hernandezlerma_runggaldier1994}) and so
\begin{equation} \label{upperfunc}
\EE[g(W^a_{t+1})] \leq \sum_{k=1}^n\sum_{i=1}^{L_k^{(n)}} \bar
q^{a,(n)}_{t+1, k,i} g(e^{(n)}_{k,i}).
\end{equation}

For the following theorem, define random variable
$$
\overline{W}_{t+1}^{a,(n)} := \sum_{k=1}^{n} \overline{W}_{t+1,k}^{a,(n)} \11\left(W_{t+1}^{a}\in\Pi^{(n)}(k)\right).
$$
Recall that $(v_{t+1}^{(n)}(p,z))_{n\in\NN}$ is a CCC sequence in $z$
for $p\in\PPP$ and that $v^{(n)}_{t+1}$ converges to
$v^*_{t+1}$. Also, recall $v_{t+1}^{(n)}(p, f_{t+1}(w,z))$ is
continuous in $w$.

\begin{lemma} \label{transConvergeUpper} If the diameter of the
  partition vanishes i.e.  $\lim_{n\to\infty} \delta^{(n)} = 0$, then
  $$
  \lim_{n\to\infty} \EE[v_{t+1}^{(n)}(p,
  f_{t+1}(\overline{W}_{t+1}^{a,(n)},z))] = \EE[v_{t+1}^*(p,
  f_{t+1}(W_{t+1}^{a},z))], \quad z\in\ZZZ.
  $$
  If, in addition,
  $\EE[v_{t+1}^{(n)}(p, f_{t+1}(\overline{W}_{t+1}^{a,(n)},z))]$ is
  convex in $z$ for $n\in\NN$, then the sequence of functions
  $(\EE[v_{t+1}^{(n)}(p,
  f_{t+1}(\overline{W}_{t+1}^{a,(n)},z))])_{n\in\NN}$ form a CCC
  sequence in $z$.
\end{lemma}
\begin{proof} 
  By construction, $\overline{W}_{t+1,k}^{a,(n)}$ takes values in the
  extreme points of $\widetilde\Pi^{(n)}(k)$ for $k=1,\dots,
  n$ and so
  $$
  \lim_{n\to\infty} \EE[ || \overline W_{t+1}^{a,(n)} - W_{t+1}^{a} ||] \leq
  \lim_{n\to\infty} \delta^{(n)} = 0.
  $$
  Thus, $\overline W_{t+1}^{a,(n)}$ converges to $W_{t+1}^{a}$ in
  distribution as $n\to\infty$. Therefore, the proof for the above
  statement then follows the same lines as the proof for Theorem
  \ref{transConvergeBounded}.
 \end{proof}

Let us define the following alternative modified transition operator:
\begin{equation} \label{upperModKern}
{\mathcal K}^{a, (n)}_{t}v(p, z) = \sum_{p'\in\PPP}
\alpha_{t+1}^{a,p,p'}\EE[v(p',f_{t+1}(\overline{W}_{t+1}^{a,(n)},z))].
\end{equation}
The next theorem establishes the uniform convergence on compact sets
of the resulting modified value functions when the new modified
transition operator \eqref{upperModKern} is used in place of the
original modified transition operator \eqref{modKernel}. Recall that
$(m_n)_{n\in\NN}$ and $(n_m)_{m\in\NN}$ are sequences of natural
numbers increasing in $n$ and $m$, respectively.
  
\begin{theorem} \label{valueConvergeUpper} Suppose
  $\lim_{n\to\infty} \delta^{(n)} = 0$. Using \eqref{upperModKern}
  gives
  \[\lim_{n\to\infty} v_t^{(m_{n},n)}(p,z) = \lim_{m\to\infty} v_t^{(m,n_m)}(p,z) = v^*_t(p,z)\]
  for $p\in\PPP$, $z\in\ZZZ$, and $t=T-1,\dots,0$. Also,
  $(v_t^{(m_{n},n)}(p,z))_{n\in\NN}$ and
  $(v_t^{(m,n_m)}(p,z))_{m\in\NN}$ both form CCC sequences in $z$ for
  all $p\in\PPP$ and $t=T-1,\dots,0$.
\end{theorem}
\begin{proof} The proof follow the same lines as the proof in Theorem
  \ref{valueConvergePW} but using Lemma
  \ref{transConvergeUpper} instead. 
\end{proof}
 
Under certain conditions, these modified expected value functions also
form a non-decreasing sequence of upper bounding functions as shown
below. The following gives analogous versions of Lemma \ref{lemmaK}
and Theorem \ref{lowerBound} but for the upper bound case.

\begin{lemma} \label{transUpper} Suppose $\Pi^{(n+1)}$ refines
  $\Pi^{(n)}$, that $v^{(n)}_{t+1}(p, f_{t+1}(w, z))$ is convex in
  $w$, and that
  $v_{t+1}^{(n)}(p,z) \geq v_{t+1}^{(n+1)}(p,z) \geq v_{t+1}^{*}(p,z)
  $ for $p\in\PPP$, $z\in\ZZZ$, and $n\in\NN$. Then
  $$
  \EE[v_{t+1}^{(n)}(p,f_{t+1}(\overline{W}_{t+1}^{a,(n)},z))] \geq
  \EE[v_{t+1}^{(n+1)}(p, f_{t+1}(\overline{W}_{t+1}^{a,(n+1)},z))]
  \geq \EE[v_{t+1}^{*}(p, f_{t+1}({W}_{t+1}^{a},z))]
  $$
  for $a\in\AAA$, $p\in\PPP$, $z\in\ZZZ$, and $n\in\NN$. 
\end{lemma}
\begin{proof} See \cite[Theorem 7.8]{hernandezlerma_runggaldier1994}
  or \cite[Theorem 1.3]{hernandezlerma_etal1995}.
 \end{proof}
  
 \begin{theorem} \label{valueUpper2} Using \eqref{upperModKern} gives for $p\in\PPP$, $z\in\ZZZ$,
   $m,n\in\NN$, and $t=0,\dots,T-1$:
  \begin{itemize}
  \item $v_t^{(m,n)}(p,z) \geq v_t^*(p,z)$ when
    $v^*_{t'}(p', f_{t'}(w, z'))$ is convex in $w$ and if
    ${\mathcal S}_{\GGG^{(m')}}h(z') \geq h(z')$ for $p'\in\PPP$,
    $z'\in\ZZZ$, $t'=1,\dots,T$, and all convex functions $h$.
  \item $v_t^{(m,n)}(p,z) \geq v_t^{(m,n+1)}(p,z)$ if $\Pi^{(n'+1)}$
    refines of $\Pi^{(n')}$ and if
    $\mathcal{S}_{\GGG^{(m')}}v^*_{T}(p', f_{T}(w, z'))$ and
    $v^{(m',n')}_{t'}(p', f_{t'}(w, z'))$ are convex in $w$ for
    $m',n'\in\NN$, $p'\in\PPP$, $z'\in\ZZZ$, and $t=1,\dots,T-1$.
  \item $v_t^{(m,n)}(p,z) \geq v_t^{(m+1,n)}(p,z)$ if
    ${\mathcal S}_{\GGG^{(m')}}h(z') \geq {\mathcal
      S}_{\GGG^{(m'+1)}}h(z')$ for $z'\in\ZZZ$, $m'\in\NN$, and all
    possible convex functions $h$.
  \end{itemize}
\end{theorem}
\begin{proof}
  The three inequalities are proven separately using backward induction.

  1) The monotonicity of \eqref{upperModKern} and Lemma
  \ref{transUpper}
  $\implies \mathcal{K}^{a,(n)}_{T-1}\mathcal{S}_{\GGG^{(m)}}
  v^*_T(p,z) \geq \mathcal{K}^{a}_{T-1}v^*_T(p,z)$
  $\implies {\mathcal S}_{\GGG^{(m)}}r_{T-1}(p,z,a) +{\mathcal
    S}_{\GGG^{(m)}}\mathcal{K}^{a,(n)}_{T-1}\mathcal{S}_{\GGG^{(m)}}
  v^*_T(p,z) \geq r_{T-1}(p,z,a) + \mathcal{K}^{a}_{T-1}v^*_T(p,z)$
  $\implies v_{T-1}^{(m,n)}(p,z) \geq v_{T-1}^*(p,z)$.  Proceeding
  inductively for $t=T-2,\dots,0$ using a similar argument as above
  gives the desired result.

  2) Lemma \ref{transUpper}
  $\implies \mathcal{K}^{a, (n)}_{T-1}\mathcal{S}_{\GGG^{(m)}}
  v^{*}_T(p,z) \geq \mathcal{K}^{a, (n+1)}_{T-1}\mathcal{S}_{\GGG^{(m)}}
  v^{*}_T(p,z)$. Therefore,
  $v^{(m,n)}_{T-1}(p,z) \geq v^{(m,n+1)}_{T-1}(p,z)$. Using the
  monotonicity of \eqref{upperModKern} and Lemma \ref{transUpper}, it
  holds that
  $$
  \mathcal{K}^{a, (n)}_{T-2}v^{(m,n)}_{T-1}(p,z) \geq \mathcal{K}^{a,
    (n)}_{T-2}v^{(m,n+1)}_{T-1}(p,z) \geq \mathcal{K}^{a,
    (n+1)}_{T-2}v^{(m,n+1)}_{T-1}(p,z)
  $$ 
  $\implies v^{(m,n)}_{T-2}(p,z) \geq v^{(m,n+1)}_{T-2}(p,z)$. Proceeding
  inductively for $t=T-3,\dots,0$ proves the second part of the statement.

  3) This can be proved via backward induction using the monotonicty
  of \eqref{upperModKern} and by the fact that
  ${\mathcal S}_{\GGG^{(m)}}r_t(p,z,a) \geq {\mathcal
    S}_{\GGG^{(m+1)}}r_t(p,z,a)$ and
  ${\mathcal S}_{\GGG^{(m)}}r_T(p,z) \geq {\mathcal
    S}_{\GGG^{(m+1)}}r_T(p,z)$ for $p\in\PPP$, $a\in\AAA$, $m\in\NN$,
  and $t=0,\dots,T-1$.
\end{proof}


\section{Numerical demonstration} \label{secNumerical}

A natural choice for $\S_{\GGG^{(m)}}$ is to use piecewise linear
functions. A piecewise linear function can be represented by a matrix
where each row or column captures the relevant information for each
linear functional. This is attractive given the availability of fast
linear algebra software.

\subsection{Bermudan put option} \label{secPut}

Markov decision processes are common in financial markets. For
example, a Bermudan put option represents the right but not the
obligation to sell the underlying asset for a predetermined strike
price $K$ at prespecified time points.  This problem is characterized
by $\mathbf{P}=\{\text{exercised}, \text{unexercised}\}$ and
$\mathbf{A}=\{\text{exercise}, \text{don't exercise}\}$. At $P_t=$
``unexercised'', applying $a=$ ``exercise'' and $a=$ ``don't
exercise'' leads to $P_{t+1}=$ ``exercised'' and $P_{t+1}=$
``unexercised'', respectively with probability one. If $P_t=$
``exercised'', then $P_{t+1}$ = ``exercised'' almost surely regardless
of action $a$. Let $\Delta$ be the time step and represent the
interest rate per annum by $\kappa$ and underlying asset price by $z$.
Defining $(z)^+ = \max(z,0)$, the reward and scrap for the option are
given by
\begin{align*}
r_{t}(\text{unexercised}, z, \text{exercise}) &= e^{-\kappa \Delta t}(K -  z)^+\\
r_{T}(\text{unexercised}, z) &= e^{-\kappa \Delta T}(K-  z)^+
\end{align*}
for all $z \in \RR_{+}$ and zero for other $p\in\PPP$ and
$a\in\AAA$. The fair price of the option is
$$ 
v^*_{0}(\text{unexercised}, z_0) = \max\left\{\EE[e^{-\kappa \Delta \tau}(K-
  Z_{\tau})^+ ]: \tau = 0,1, \dots, T\right\}.
$$
The option is assumed to reside in the Black-Scholes world
where the asset price process $(Z_{t})_{t=0}^{T}$ follows geometric
Brownian motion i.e.
\begin{equation*}
  Z_{t+1}=  W_{t+1} Z_t = e^{(\kappa - \frac{\text{vol}^2}{2})\Delta + \text{vol}\sqrt{\Delta}N_{t+1}} Z_{t}
\end{equation*}
where $(N_{t})_{t=1}^{T}$ are independent standard normal random
variables and $\text{vol}$ is the volatility of stock returns. Note
that the disturbance is not controlled by action $a$ and so the
superscript is removed from $W^a_{t+1}$ for notational simplicity in
the following subsections.  The reward and scrap functions are convex
and Lipschitz continuous in $z$. It is not hard to see that under the
linear state dynamic for $(Z_t)_{t=0}^T$, the resulting expected value
functions and value functions are also convex, Lipschitz continuous,
and decreasing in $z$. In the following two subsections, two different
$\mathcal{S}_{\GGG^{(m)}}$ schemes are used to approximate these
functions.

\subsection{Approximation using tangents} \label{secTangents}

It is well known that a convex real valued function is differentiable almost
everywhere \cite[Theorem 25.5]{rockafellar} and so can be approximated
accurately on any compact set using a sufficient number of its
tangents. Suppose that convex function $h:\ZZZ \rightarrow \RR$ holds
tangents on each point in $\GGG^{(m)}$ given by
$\{h_1'(z), \dots, h_m'(z)\}$ and that the approximation scheme
${\mathcal S}_{\GGG^{(m)}}$ takes the maximising tangent to form a
convex piecewise linear approximation of $h$ i.e.
$$
{\mathcal S}_{\GGG^{(m)}}h(z) = \max\{h_{1}'(z),\dots, h_m'(z)\}.
$$  
It is not hard to see that the resulting approximation
${\mathcal S}_{\GGG^{(m)}}h$ is convex, piecewise linear, and
converges to $h$ uniformly on compact sets as $m\to\infty$. It is also
clear that
${\mathcal S}_{\GGG^{(m)}}h(z) \leq {\mathcal S}_{\GGG^{(m+1)}}h(z)
\leq h(z)$ for all $z\in\ZZZ$. Note that while the choice of a tangent
may not be unique at some set of points in $\ZZZ$, the uniform
convergence on compact sets of this scheme is not affected.
Assumption \ref{assSampleConvex} holds for any choice of grid and
disturbance sampling under the linear state dyanmics in Section
\ref{secPut}. To see this, note that if $v(p,z)$ is any function
convex in $z$ then $v(p,wz)$ is also convex in $z$ for any $w\in\WWW$.

As a demonstration, a Bermudan put option is considered with strike
price 40 that expires in 1 year. The put option is exercisable at 51
evenly spaced time points in the year, which includes the start and
end of the year. The interest rate and the volatility is set at $0.06$
p.a. and $0.2$ p.a., respectively. Recall that there are two natural
choices for the sampling $(W^{(n)}_{t+1}(k))_{k=1}^{n}$. The first
choice involves partitioning the disturbance space $\WWW = \RR_+$ and
second choice involves obtaining the disturbance sampling randomly via
Monte Carlo. Using the first approach, the space $\WWW = \RR_+$ is
partitioned into sets of equal probability measure and then the
conditional averages on each component are used for the sampling. The
use of anti-thetic disturbances in Monte Carlo sampling scheme is
found to lead to more balanced estimates. While it is clear that the
second choice of disturbance sampling is easier to implement, the
convergence seems to be slower as demonstrated by Figure
\ref{plotCompareDisturb}. In Figure \ref{plotCompareDisturb}, the size
of the disturbance sampling is varied from $1000$ to $10000$ in steps
of $1000$ under the two different sampling choices. The grid is given
by $401$ equally spaced grid points from $z=20$ to $z=60$. In the left
plot of Figure \ref{plotCompareDisturb}, the conditions in Theorem
\ref{lowerBound} are satisfied and so the estimates give a lower bound
for the true value. This lower bound increases with the size of the
sampling as proved in Thereom \ref{lowerBound}.  Despite the
difference in convergence speeds, Figure \ref{plotCompareValue} shows
that the two sampling schemes return very similar option value
functions for sufficently large disturbance samplings. The ticks on
the horizontal axis indicate the location of grid points. The curve in
the left plot gives a lower bounding function for the true value
function.

\begin{figure}[h!]
  \centering
  \includegraphics[height=2in,width=0.45\textwidth]{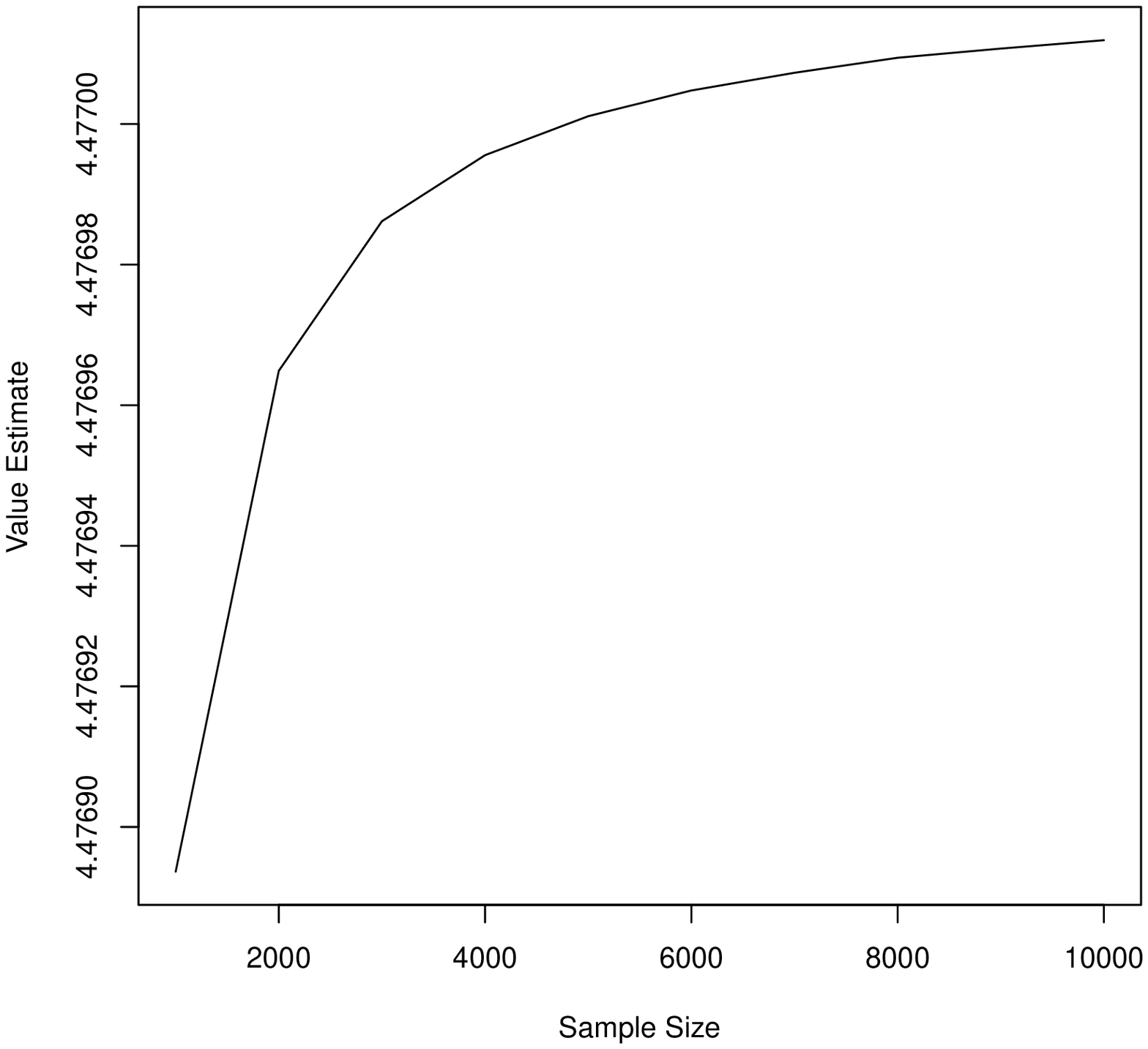}
  \includegraphics[height=2in,width=0.45\textwidth]{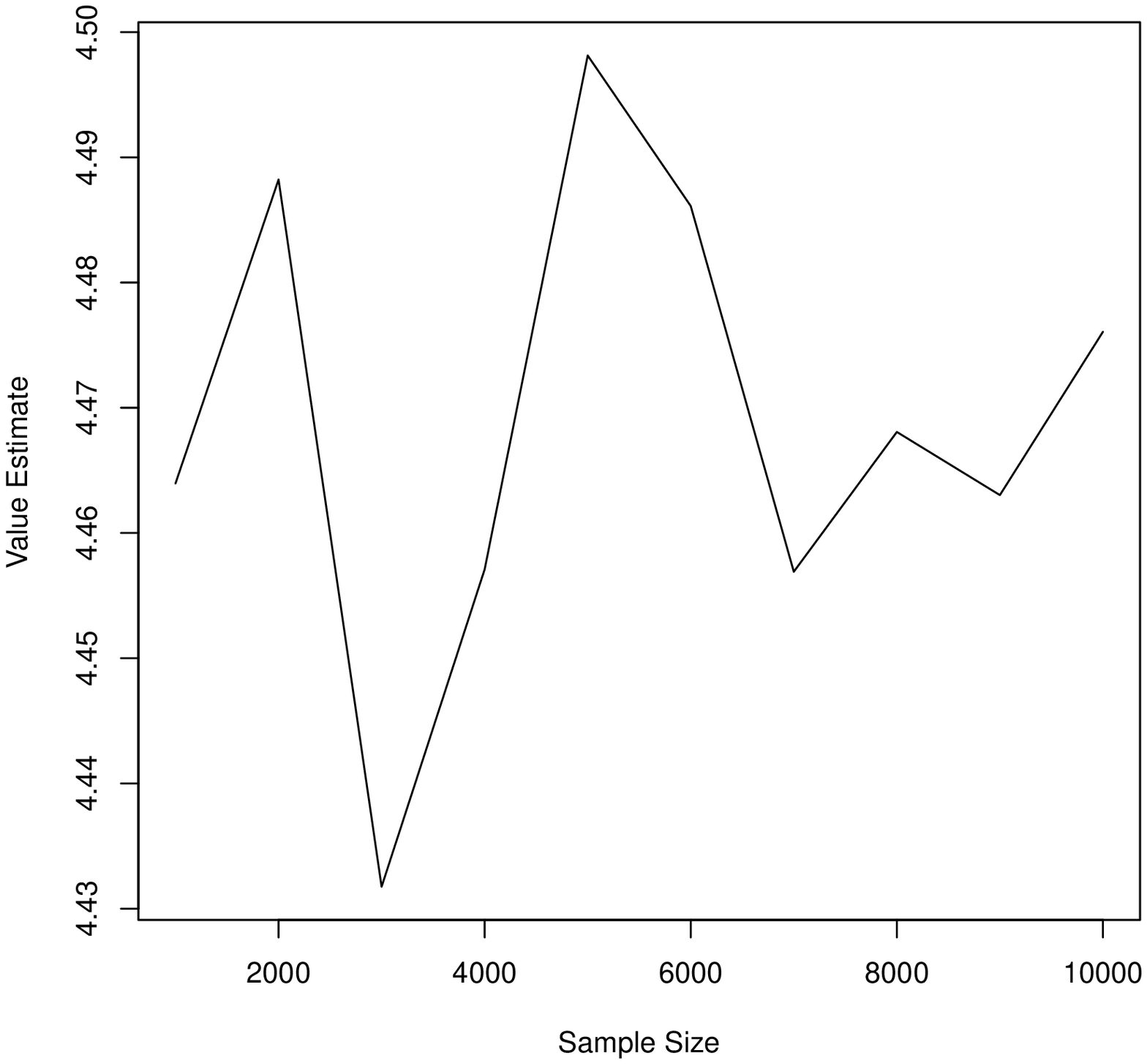}
  \caption{Estimate as sampling size is increased. Local
    averages (left) and Monte Carlo (right).}
  \label{plotCompareDisturb}
\end{figure}

\begin{figure}[h]
  \centering
  \includegraphics[height=2in,width=0.45\textwidth]{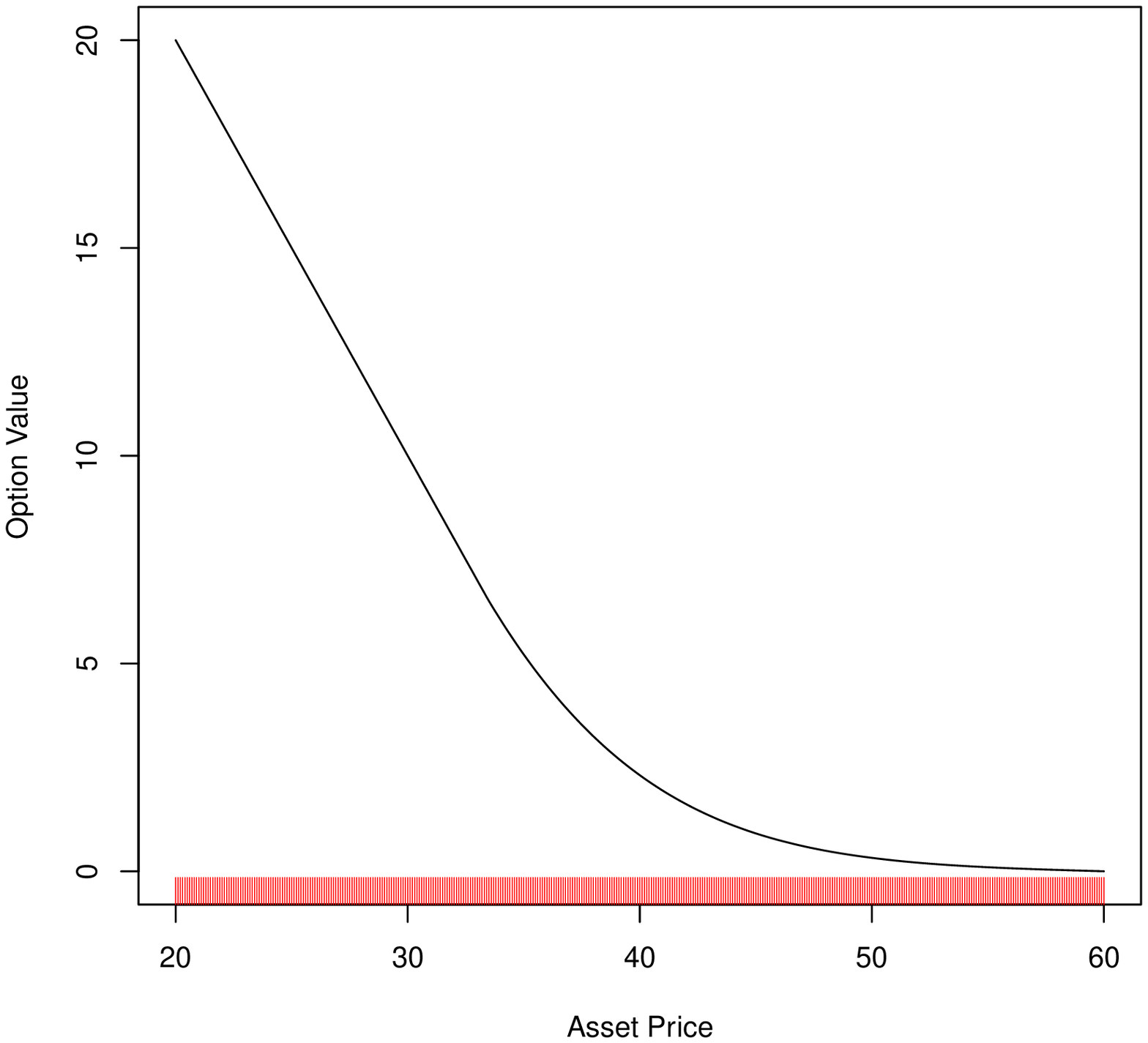}
  \includegraphics[height=2in,width=0.45\textwidth]{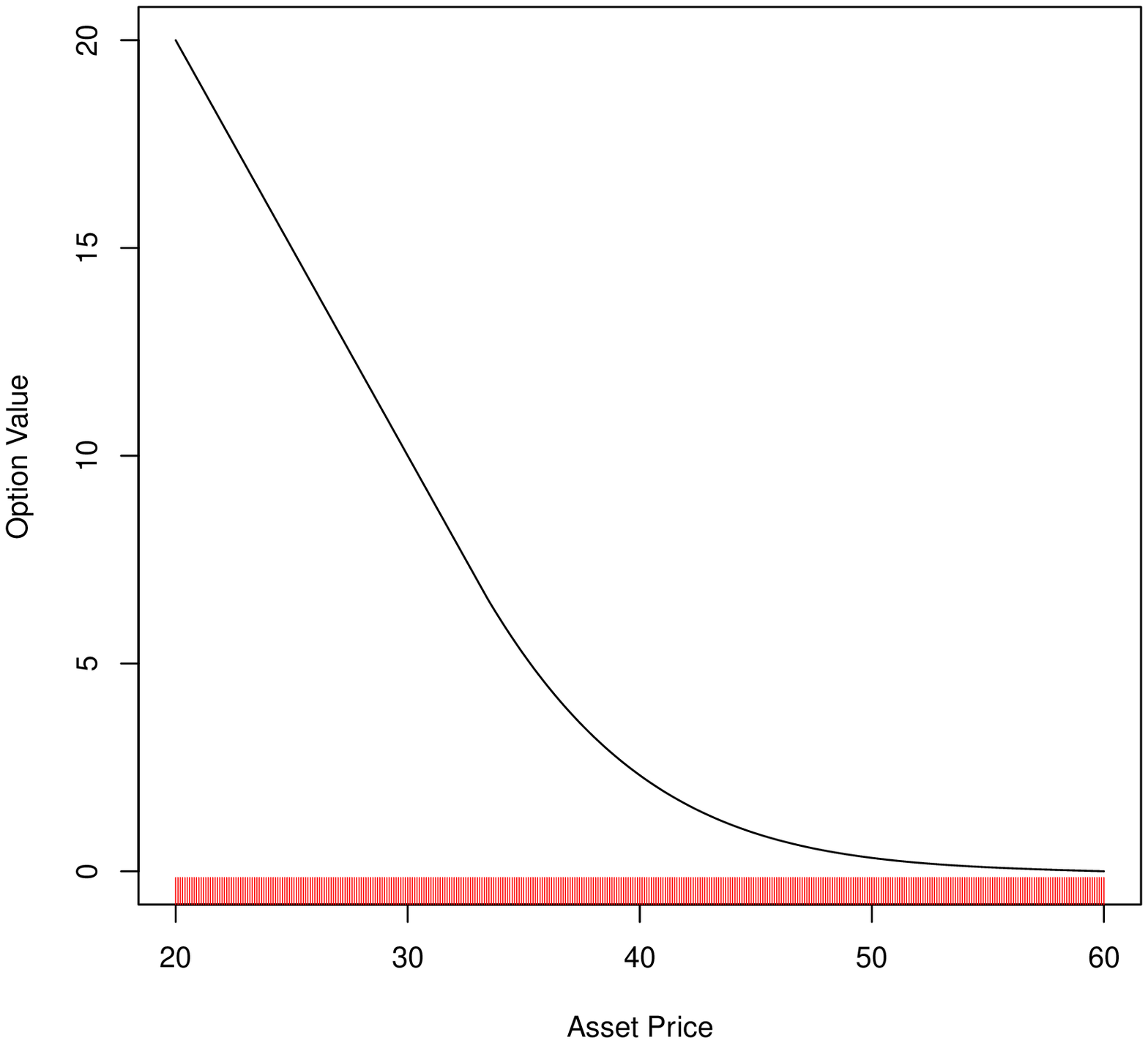}
  \caption{Function approximation with $n=10000$. Local averages
    (left) and Monte Carlo (right).}
  \label{plotCompareValue}
\end{figure}

\subsection{Upper bound via linear interpolation}

This subsection constructs upper bounding functions for the option
value using the approach presented in Section \ref{secUpper}.  While
the support of the distribution of $W_t$ is unbounded, it can be
approximated by a compact set $\overline{\WWW}$ containing
$99.9999999\%$ of the probability mass i.e.
$\PP(W_t \in \overline{\WWW}) = 0.999999999$ for all $t=1,\dots,T$. To
this end, introduce the truncated distribution $\overline\PP$ defined
by $\overline\PP(W_t\in\mathbf{B}) = \alpha \PP(W_t\in\mathbf{B})$ for
all $\mathbf{B}\subseteq\overline\WWW$ where
$\alpha = 1/\PP(W_t \in \overline{\WWW})$ is the normalizing constant.
In the following results, set $\overline{\WWW} = [0.841979,1.18958]$
and use disturbance space partitions consisting of $n$ convex
components of equal probability measure. Suppose that the extreme
points are ordered $e^{(n)}_{k,1} < e^{(n)}_{k,2}$ for all
$k=1,\dots,n$. For $k=1,\dots, n$, define points
$e^{(n)}_k = e^{(n)}_{\floor{(k+1)/2}, \floor{k+2-2\floor{(k+1)/2}}}$
and $e^{(n)}_{n+1} = e^{(n)}_{n,2}$ where $\floor{\hspace{1mm}}$
denotes the integer part. Defining partial expectations
$\Lambda_{t}(a,b) = \EE^{\overline\PP}[W_t \11(W_t \in[a,b])]$ and
using
$$
q^{(n)}_{t,k,1}(w) = \frac{e_{j,2} - w}{e_{j,2}-e_{j,1}} 
\quad \text{and} \quad
q^{(n)}_{t,k,2}(w) = \frac{w - e_{j,1}}{e_{j,2}-e_{j,1}} 
$$
one can determine
$$
\overline\PP\left(\overline{W}_{t}^{(n)} = e_1^{(n)}\right)
= \frac{\left(e^{(n)}_2/n - \Lambda_t(e^{(n)}_1,e^{(n)}_2) \right)}
{e^{(n)}_2-e^{(n)}_1}, 
$$
$$
\overline\PP\left(\overline{W}_{t}^{(n)} = e_{n+1}^{(n)}\right)
= \frac{\left(\Lambda_t(e^{(n)}_{n},e^{(n)}_{n+1}) - e^{(n)}_{n}/n \right)}
{e^{(n)}_{n+1}-e^{(n)}_n}, 
$$
$$
\overline\PP\left(\overline{W}_{t}^{(n)} = e_{j}^{(n)}\right) =
\frac{\left(e^{(n)}_{j+1}/n - \Lambda_t(e^{(n)}_{j},e^{(n)}_{j+1})
  \right)} {e^{(n)}_{j+1}-e^{(n)}_j} + \frac{\left(
    \Lambda_t(e^{(n)}_{j-1},e^{(n)}_{j}) - e^{(n)}_{j-1}/n \right)}
{e^{(n)}_{j}-e^{(n)}_{j-1}}
$$
for $j=2,\dots,n$.

Recall that the reward, scrap, and the true value functions are
Lipschitz continuous, convex, and non-increasing in $z$. Therefore,
there exists a $z'$ such that these functions are linear in $z$ when
$z < z'$ fro some $z'$. In fact, it is well known that
$v_t^*(\text{unexercised},z) =
r_t(\text{unexercised},z,\text{exercise})$ when $z < z'$ for
$t=0,\dots,T$ and some $z'$. Let us define our approximation scheme in
the following manner. Suppose $h:\ZZZ\to\RR$ is a convex function
non-increasing in $z$ and $h(z)=h'(z)$ when $z<z'$.  For
$\GGG^m = \{g^{(1)},\dots, g^{(m)}\}$ where $g^{(1)} \leq z'$ and
$g^{(1)}<g^{(2)}<\dots<g^{(m)}$, set
$$
\S_{\GGG^{(m)}}h(z) =
 \Bigg\{ \begin{array}{ll}
        h'(z) & \mbox{if $z \leq g^{(1)}$};\\
        d_i (z - g^{(i)}) +   h(g^{(i)}) & \mbox{if $g^{(i)} < z \leq g^{(i+1)}$};\\
        h(g^{(m)}) & \mbox{if $z > g^{(m)}$},\end{array} 
$$
where $d_i =\frac{h(g^{(i+1)}) - h(g^{(i)})}{g^{(i+1)} - g^{(i)}} $
for $i=2,\dots,m-1$. For $g^{(1)} \leq z \leq g^{(m)}$,
$\S_{\GGG^{(m)}}$ forms an approximation of $h$ via linear
interpolation. It is not hard to see that
$\S_{\GGG^{(m)}}h(z) \geq \S_{\GGG^{(m+1)}}h(z) \geq h(z)$ for all
$z\in\ZZZ$. It is also clear that Assumption \ref{assSampleConvex}
will hold for any grid and disturbance sampling due to the linear
dynamics of $Z_t$.  

\begin{figure}[h!]
  \centering
  \includegraphics[height=2in,width=0.45\textwidth]{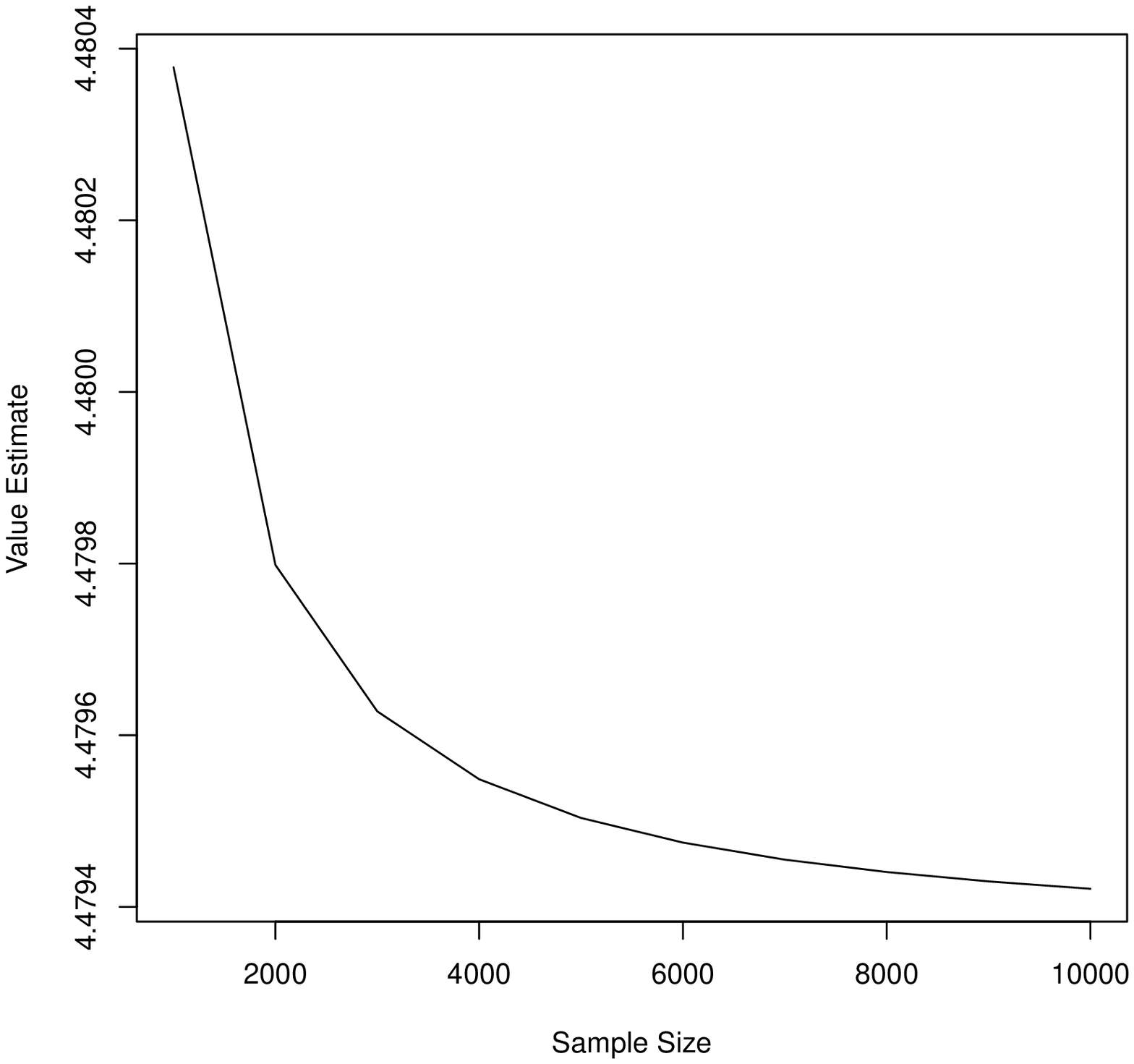}
  \includegraphics[height=2in,width=0.45\textwidth]{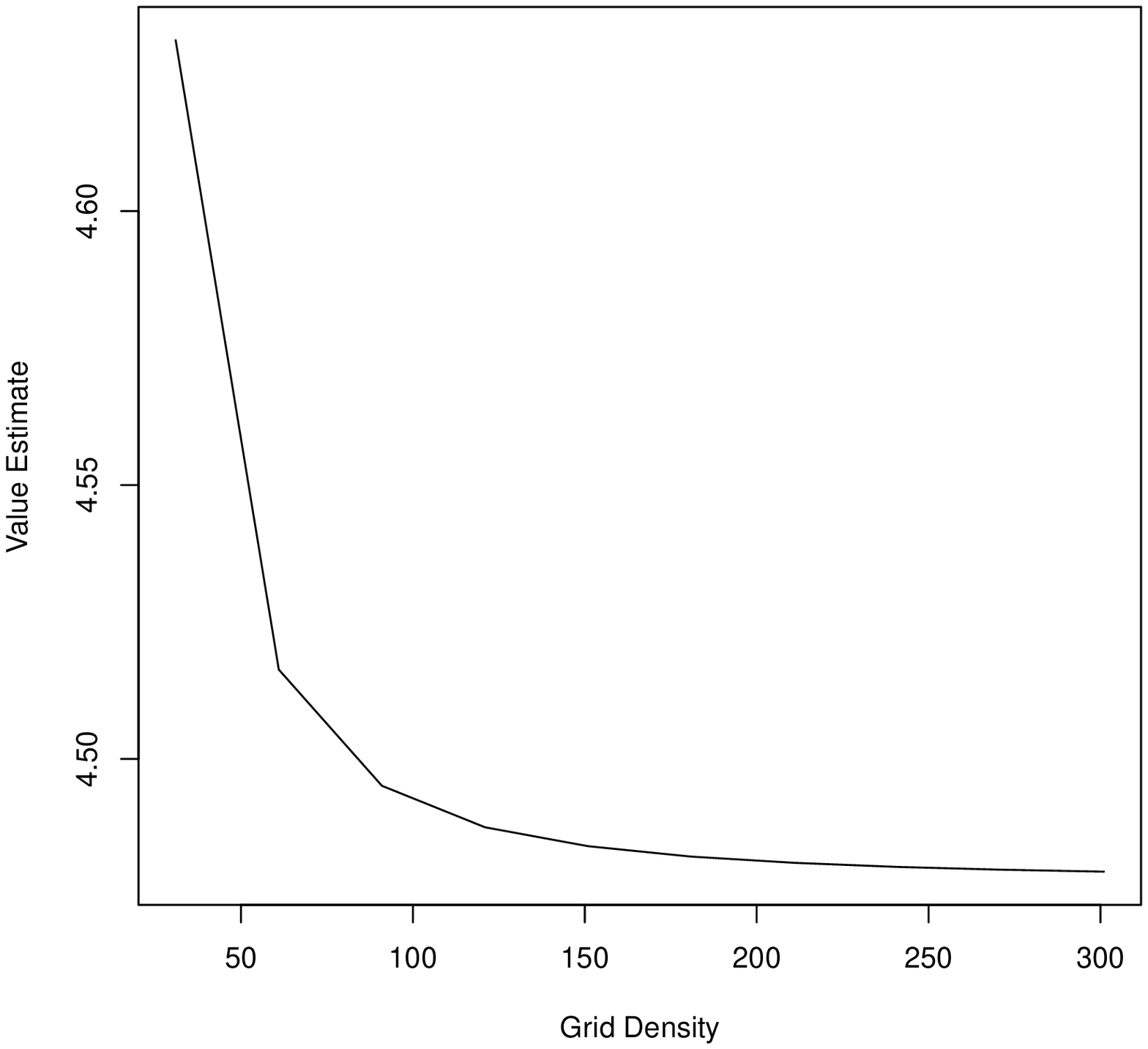}
  \caption{Price estimate 
    as disturbance sampling (left plot) or grid density (right plot)
    is increased.}
  \label{plotCompareUpper}
\end{figure}

Using the same put option paramters as before and
a grid with points equally spaced between $z=30$ and $60$, Figure
\ref{plotCompareUpper} examines the behaviour of the upper bound for
the option with $Z_0=36$ as we increase the size of the partition $n$
or the grid density $m$. In the left plot, the size of the partition
$n$ is increased for fixed $m=301$. Similarly, the grid density is
increased in the right plot for fixed $n=10000$. Observe that the
bounds is decreasing in both $m$ and $n$ as proved by Theorem
\ref{valueConvergeUpper}.

\subsection{Accuracy and speed}
Table \ref{tablePut} gives points on the lower and upper bounding
functions for the fair price of the option at different starting asset
prices and expiry dates. Columns $2$ to $4$ give an option with the 1
year expiry. A grid of $301$ equally spaced points from $z=30$ to
$z=60$ is used. Columns $5$ to $7$ gives an option expiring in 2 years
and is exerciseable at 101 equally spaced time points including the
start. A wider grid of $401$ equally spaced points from $z=30$ to
$z=70$ is used to account for the longer time horizon. For both, a
disturbance partition of size $n=1000$ is used. The other option
parameters (e.g. interest rate) remain the same as before. Tighter
bounding functions can be obtained by increasing the grid density $m$
or size of the disturbance sampling $n$. This is illustrated by
columns $8-10$ where we revisit the 1 year expiry option with $4001$
equally spaced grid points from $z=30$ to $z=70$ and a disturbance
partition of $n=20000$ is used.

\begin{table}[h]
  \caption{Bermuda put valuation with $vol=0.2$.
    \label{tablePut}}
  \centering
 \setlength\tabcolsep{3.5pt} 
  \begin{tabular}{l|ccc|ccc|ccc}
    &\multicolumn{3}{c}{Expiry 1 year}&\multicolumn{3}{c}{Expiry 2 years} &\multicolumn{3}{c}{Dense, Expiry 1 year}\\
    $Z_0$ & Lower & Upper & Gap & Lower & Upper & Gap & Lower & Upper & Gap\\
    \hline
    32 & 8.00000 & 8.00000 & 0.00000 & 8.00000 & 8.00000 & 0.00000 & 8.00000 & 8.00000 & 0e+00 \\
    34 & 6.05155 & 6.05318 & 0.00163 & 6.22898 & 6.23254 & 0.00356 & 6.05198 & 6.05201 & 2e-05 \\
    36 & 4.47689 & 4.48038 & 0.00348 & 4.83885 & 4.84435 & 0.00550 & 4.47780 & 4.47785 & 5e-05 \\
    38 & 3.24898 & 3.25347 & 0.00450 & 3.74319 & 3.74964 & 0.00645 & 3.25011 & 3.25018 & 7e-05 \\
    40 & 2.31287 & 2.31766 & 0.00479 & 2.88294 & 2.88965 & 0.00670 & 2.31405 & 2.31413 & 8e-05 \\
    42 & 1.61582 & 1.62047 & 0.00465 & 2.21077 & 2.21735 & 0.00658 & 1.61696 & 1.61704 & 8e-05 \\
    44 & 1.10874 & 1.11311 & 0.00437 & 1.68826 & 1.69456 & 0.00630 & 1.10985 & 1.10993 & 8e-05 \\
    46 & 0.74795 & 0.75217 & 0.00423 & 1.28419 & 1.29023 & 0.00604 & 0.74915 & 0.74922 & 7e-05 \\
  \end{tabular}
\end{table}

The results in Table \ref{tablePut} were computed using the R script
provided in the appendix.
On a Linux Ubuntu 16.04 machine with Intel i5-5300U CPU @2.30GHz and
16GB of RAM, it takes roughly 0.15 cpu seconds to generate each of the
bounding functions represented by columns $2$ and $3$. For expiry 2
years, it takes around 0.20 cpu seconds to generate each bounding
function.  For columns $8$ or $9$, it takes around $40$ cpu seconds.
Note that the numerical methods used are highly parallelizable. The
code uses some multi-threaded code and the times can be reduced to
between 0.03 - 0.10 real world seconds to generate columns $2,3,5$,
and $6$ each on four CPU cores. For columns $8$ and $9$, the times are
reduced to around $10$ real world seconds.  Please note that faster
computational times can be attained by a more strategic placement of
the grid points. The reader is invited to replicate these results
using the $R$ script provided in the appendix. Now given the excellent
quality of the value function approximations, the optimal policy can
then be obtained via \eqref{optimal_policy}. This is demontrated by
Figure \ref{plotComparePut} where the right plot gives the optimal
exercise boundary. If at any time the price of the underlying asset is
below the cutoff, it is optimal to exercise the option.

\begin{figure}[h!]
  \centering
  \includegraphics[height=2in,width=0.45\textwidth]{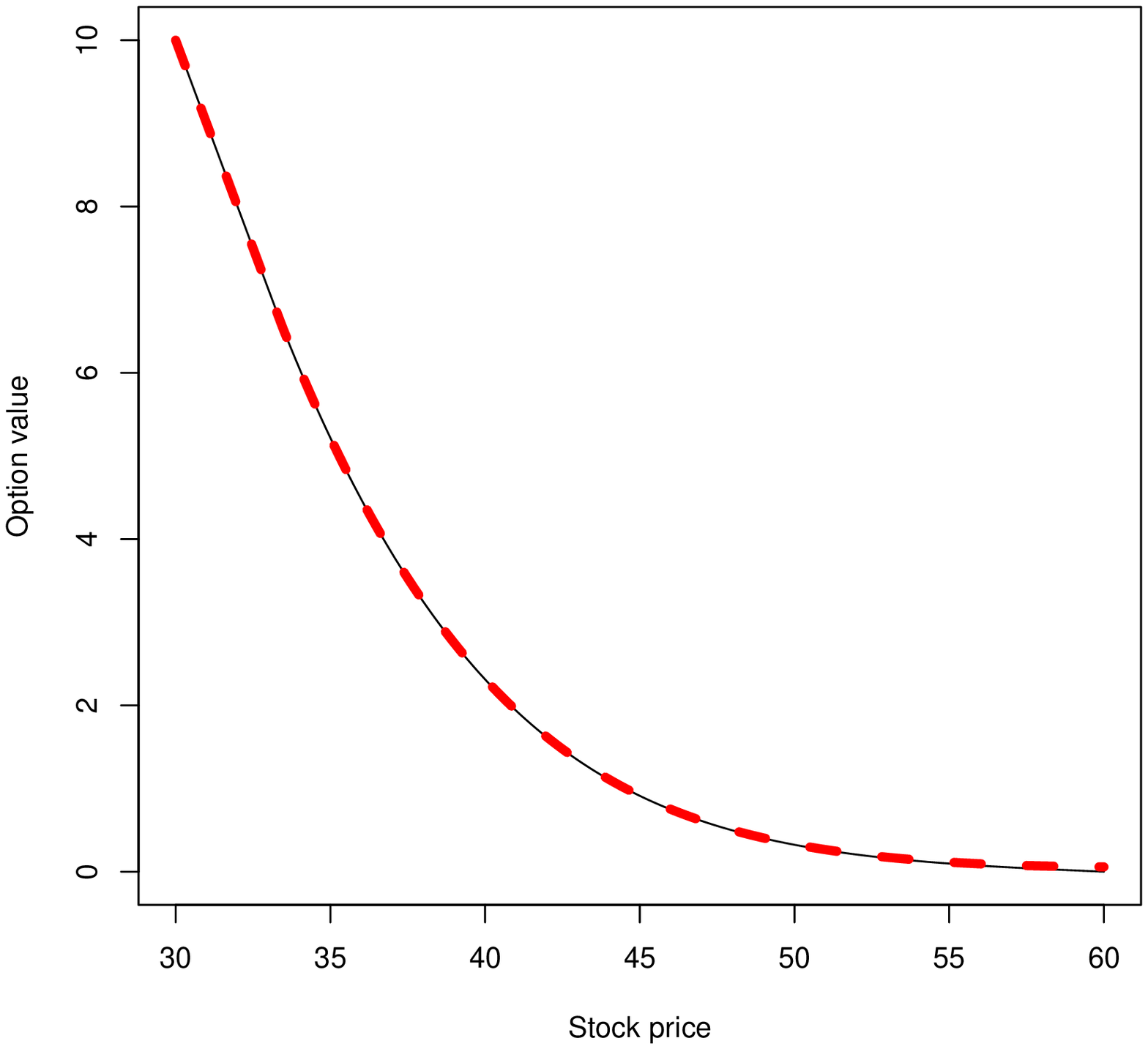}
  \includegraphics[height=2in,width=0.45\textwidth]{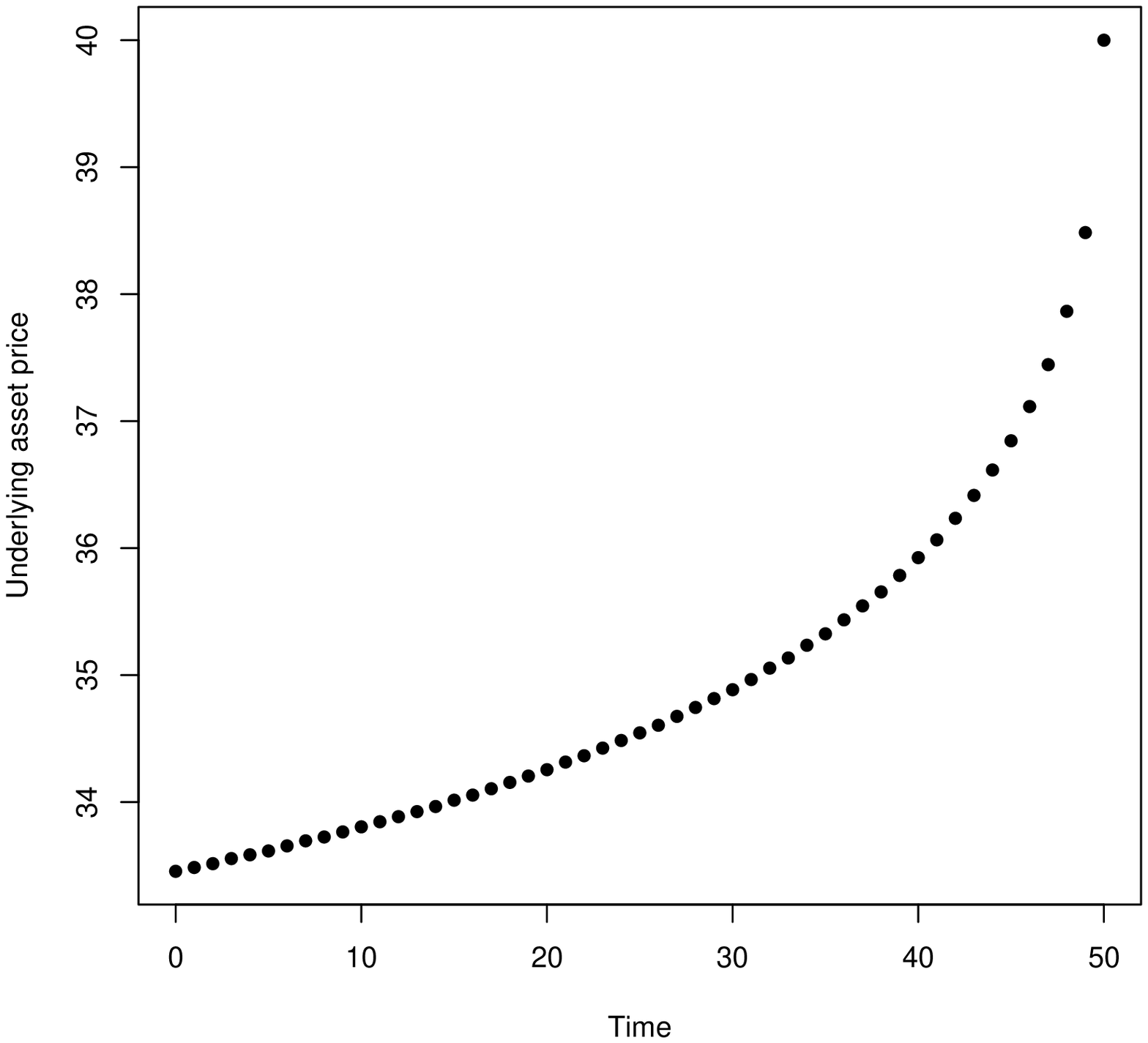}
  \caption{Left plot gives the lower and upper bounding functions for
    option price with $vol=0.2$ and $1$ year expiry. The dashed lines
    indicate the upper bound while the unbroken curves give the lower
    bounds. The optimal exercise boundary is given in the right plot.}
  \label{plotComparePut}
\end{figure}

\begin{figure}[h!]
  \centering
  \includegraphics[height=2in,width=0.45\textwidth]{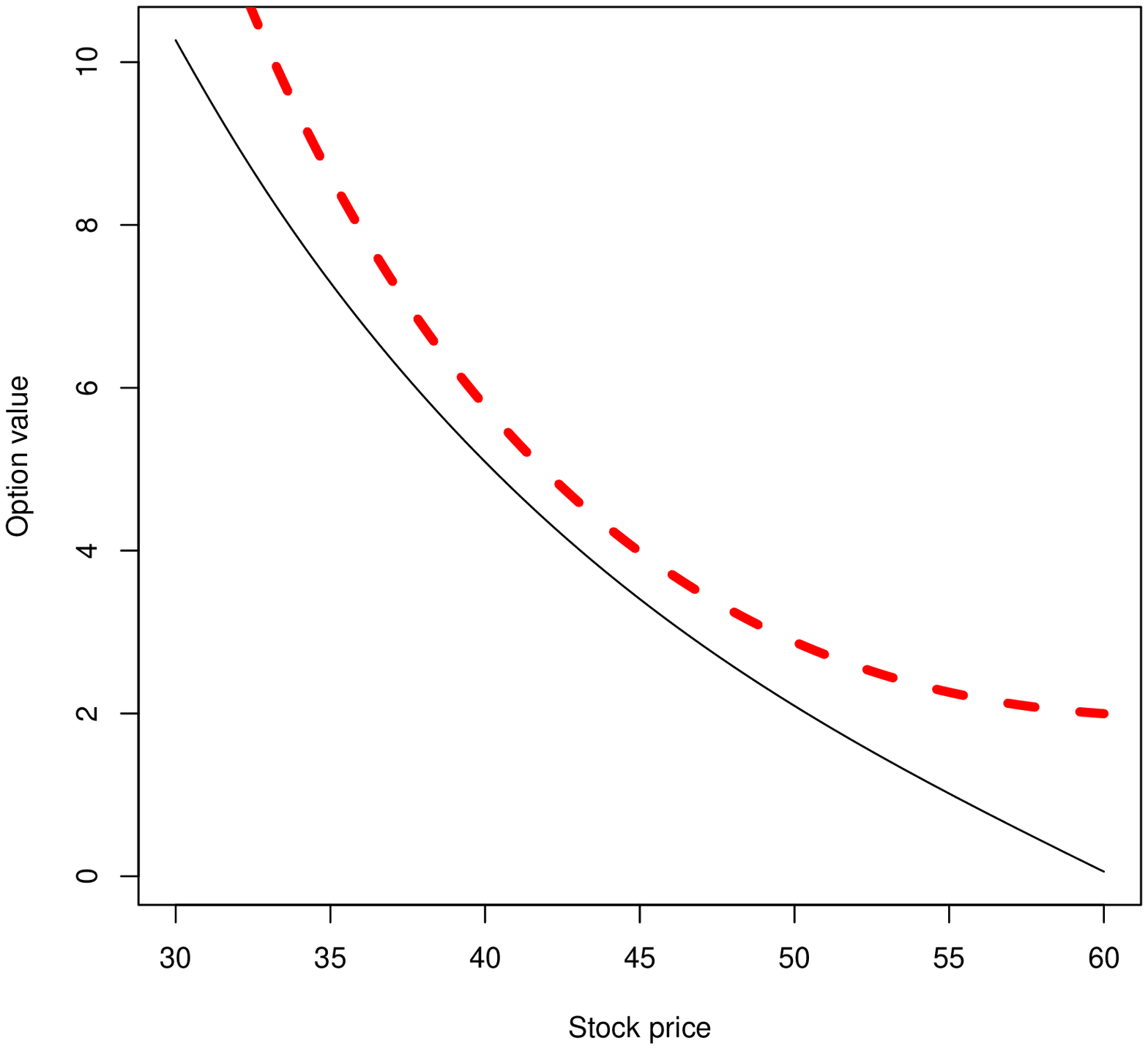}
  \includegraphics[height=2in,width=0.45\textwidth]{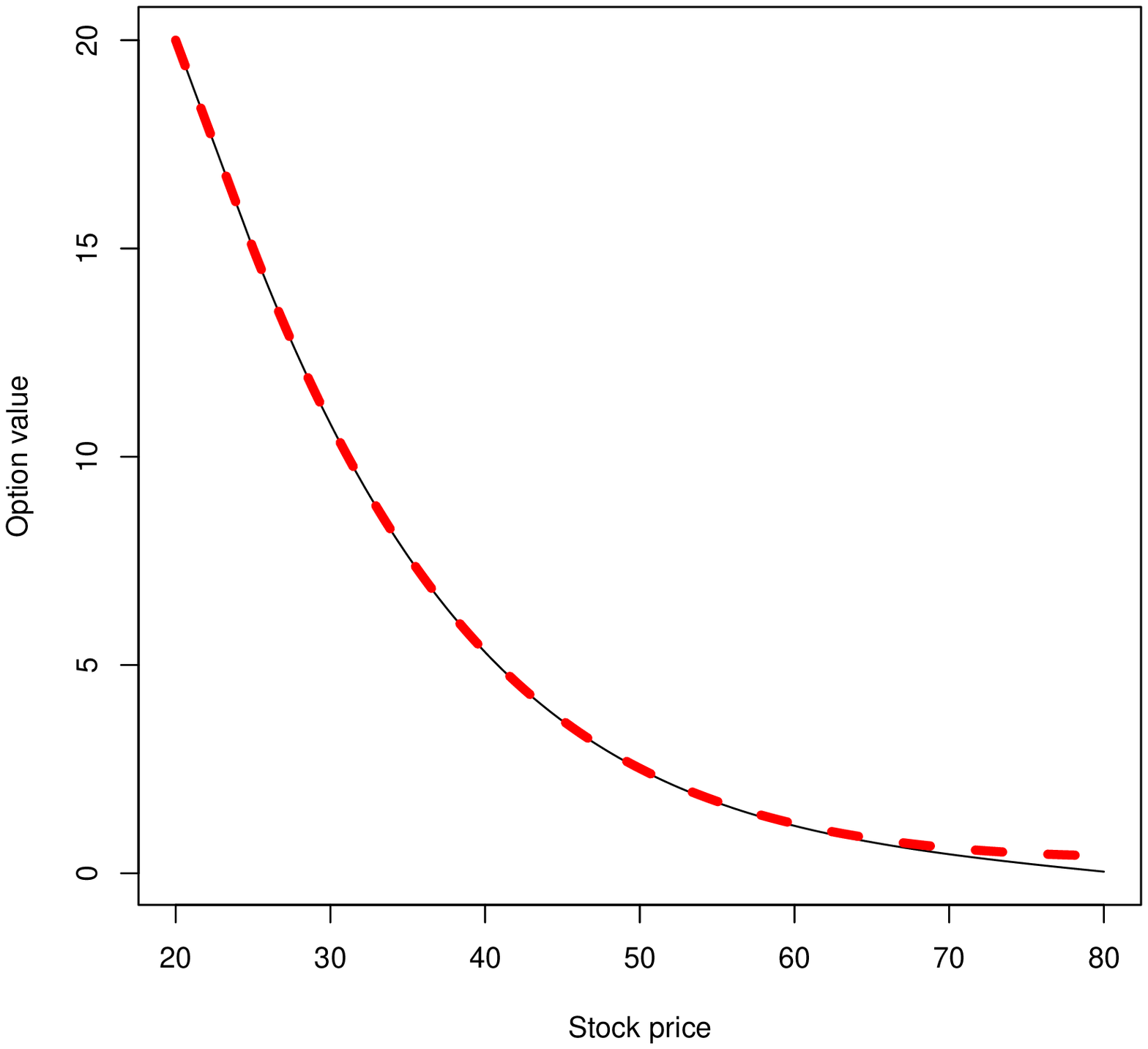}
  \caption{Lower and upper bounding functions for option price with 1
    year expiry and $vol=0.4$. The left plot uses the same grid as
    before while the right uses the wider grid. The dashed lines
    indicate the upper bound while the unbroken curves give the lower
    bound.}
  \label{plotComparePut2}
\end{figure}

For the 1 year expiry option, suppose the volatitilty is doubled to
$vol = 0.4$. Using the same grid and disturbance sampling from columns
2 and 3 leads to poor bounding functions as illustrated by the left
plot of Figure \ref{plotComparePut2}. However, spreading the same
number of grid points evenly between $z=20$ and $z=80$ leads to better
bounds as shown by the right plot. The computational times remain the
same as before. Therefore, the tightness of the bounds can be achieved
by the simple modification of either the grid or disturbance sampling
and their adjustments can be done independently of each other as shown
in Theorem \ref{valueConvergePW} and Theorem \ref{valueConvergeUpper}.
This is in contrast to the popular least squares Monte carlo method
where the size of the regression basis should not grow independently
of the number of simulated paths
\cite{glasserman_yu2004}. 

\section{Conclusion} \label{secConclusion}

This paper studies the use of a general class of convex function
approximation to estimate the value functions in Markov decision
processes. The key idea is that the original problem is approximated
with a more tractable problem and under certain conditions, the
solutions to this modified problem converge to their original
counterparts.  More specifically, this paper has shown that these
approximations may converge uniformly to their true unknown
counterparts on compact sets under different sampling schemes for the
driving random variables. Exploiting further conditions leads to
approximations that form either a non-decreasing sequence of lower
bounding functions or a non-increasing sequence of upper bounding
functions. Numerical results then demonstrate the speed and accuracy
of a proposed approach involving piecewise linear functions. While the
focus of this paper has been numerical work, one can in principle
replace the original problem with a more analytically tractable
problem and obtain the original solution by considering the limits.

The starting state $X_0 = x_0$ for the decision problem was assumed to
be known with certainity. Suppose this is not true and $X_0$ is
distributed with distribution $\PP_{X_0}$. Then, the value function
for this case can be obtained simply via
$$
\int_{\XXX} v^*_0(x')\PP_{X_0}(\mathrm{d} x')
$$
where $v_0^*$ is obtained assuming the starting state is known.  Now
note that the insights presented in this paper can be adapted to
problems where the functions in the Bellman recursion are not convex
in $z$. For example, it is not hard to see that they can be easily
modified for minimization problems involving concave functions. That
is, problems of the form
$$
{\mathcal T}_{t}v(x)=\min_{a \in \AAA} (r_{t}(x,a)+ {\mathcal
  K}^{a}_{t}v(x))
$$
where the scrap and reward functions are concave in $z$ and the
transition operator preserves concavity. To see this, note that the
sum of concave functions is concave and the pointwise minimum of
concave functions is also concave.  Finally, the methods shown in this
paper has been adapted to infinite time horizon contracting Markov
decision processes in \cite{yee_infinite}. Extensions to partially
obeservable Markov decision processes will be considered in future
research.

\section*{Acknowledgements}
The author would like to thank Juri Hinz, Onesimo Hernandez-Lerma, and
Wolfgang Runggaldier for their help.

\appendix
\section{R Script for Table \ref{tablePut}}

The script (along with the R package) used to generate columns 2, 3,
and 4 in Table \ref{tablePut} can be found at
\url{https://github.com/YeeJeremy/ConvexPaper}.  To generate the
others, simply modify the values on Line 6, Line 10, Line 11, and/or
Line 14.

\bibliography{main}
\bibliographystyle{amsplain}

\end{document}